\documentclass[reqno]{amsart}

	\usepackage[utf8]{inputenc}

	\usepackage{amsmath,amsfonts,amssymb,amsthm}

	\usepackage[T1]{fontenc}

	\usepackage{etoolbox}
	\patchcmd{\section}{\scshape}{\scshape\bfseries}{}{}
	\makeatletter
	\renewcommand{\@secnumfont}{\scshape\bfseries}
	\makeatother

	\usepackage{parskip}

	\usepackage[bookmarks=true,bookmarksopen=true]{hyperref}

	\theoremstyle{theorem}
	\newtheorem{theorem}{Theorem}

	\newtheorem*{question}{Question}
	\newtheorem{corollary}[theorem]{Corollary}
	\newtheorem{lemma}[theorem]{Lemma}
	\newtheorem{proposition}[theorem]{Proposition}
	
	\numberwithin{theorem}{section}
	
	\theoremstyle{plain}
	\newtheorem{result}{Theorem}
	\newtheorem{propositionresult}[result]{Proposition}

	\theoremstyle{definition}
	\newtheorem*{definition}{Definition}
	\newtheorem{example}[theorem]{Example}
	
	\theoremstyle{remark}
	\newtheorem{remark}[theorem]{Remark}
	
	\numberwithin{equation}{section}
	
	\numberwithin{table}{section}

	\usepackage{tikz}
	\usetikzlibrary{matrix}

	\usepackage{todonotes}

	\usepackage{enumerate}

	\usepackage{float}

	\makeatletter
	\def\blfootnote{\gdef\@thefnmark{}\@footnotetext}
	\makeatother

	\DeclareMathOperator{\Ric}{Ric}
	\DeclareMathOperator{\symrank}{symrank}

	\DeclareMathOperator{\Image}{Im}

	\DeclareMathOperator{\grad}{grad}
	
	\DeclareMathOperator{\Hess}{Hess}

	\DeclareMathOperator{\trace}{trace}
	\DeclareMathOperator{\codim}{codim}
	
	\renewcommand{\index}{\mathrm{index}\,}

	\mathcode`l="8000
	\begingroup
	\makeatletter
	\lccode`\~=`\l
	\DeclareMathSymbol{\lsb@l}{\mathalpha}{letters}{`l}
	\lowercase{\gdef~{\ifnum\the\mathgroup=\m@ne \ell \else \lsb@l \fi}}%
	\endgroup

	\newcommand*{\defeq}{\mathrel{\vcenter{\baselineskip0.5ex \lineskiplimit0pt
				\hbox{\scriptsize.}\hbox{\scriptsize.}}}%
		=}

	\title[Torus actions \& $\Ric_k>0$]{Torus actions on manifolds with positive intermediate Ricci curvature}
	
	\author{Lawrence Mouill\'e}
	
	\address{Department of Mathematics \\ Syracuse University \\ Syracuse, NY USA.}
	
	\email{lgmouill@syr.edu}
	
	\date{\today}

\begin{document}

	\begin{abstract}
		We study closed, simply connected manifolds with positive $2^\mathrm{nd}$-intermediate Ricci curvature and large symmetry rank.
		In odd dimensions, we show that they are spheres.
		In even dimensions other than $6$, we show that they  must have positive Euler characteristic.
		Under stronger assumptions on the symmetry rank, we show that such even dimensional manifolds must have trivial odd degree integral cohomology, and if the second Betti number is no more than $1$, they are either spheres or complex projective spaces.
		In the process, we establish new tools for studying isometric actions on closed manifolds with positive $k^\mathrm{th}$-intermediate Ricci for values of $k \geq 2$.
		These tools include generalizations of the isotropy rank lemma, symmetry rank bound, and connectedness principle from the setting of positive sectional curvature.

	\end{abstract}

	\maketitle
	
	\blfootnote{\textup{2020} \textit{Mathematics Subject Classification}: 53C20 (Primary), 57S25, 57N65, 58E10 (Secondary)}
	
	\blfootnote{The author was funded in part by NSF Grant DMS - 1612049 and NSF Grant DMS - 2202826.}
	
	\section{Introduction}

	On an $n$-dimensional manifold, positive $k^\mathrm{th}$-intermediate Ricci curvature is a condition interpolating between positive sectional curvature ($k = 1$) and positive Ricci curvature ($k = n - 1$). 
	There has recently been increased interest in studying manifolds with lower bounds on $k^\mathrm{th}$-intermediate Ricci curvature.
	For example, many results from the setting of positive or non-negative sectional curvature have been adapted to this setting, including generalizations of the Synge theorem and Weinstein fixed point theorem by Wilhelm \cite{wilhelminterm}, the Heintze-Karcher inequality by Chahine \cite{yousef}, and the Gromoll-Meyer theorem and Cheeger-Gromoll soul theorem for open manifolds by Shen \cite{shenkth}.
	In addition, many comparison results have been proven by Guijarro and Wilhelm in their series of papers \cite{focalradius,restrictions,softconnprinc}.
	In this article, to further the study of positive intermediate Ricci curvature, we take inspiration from the Grove symmetry program.
	
	In the setting of positive sectional curvature, the Grove symmetry program has proven to be a great source of insights.
	The goal of this program is to classify positively curved manifolds that have ``large isometry groups.''
	Karsten Grove initiated this investigation in 1992, motivated by a result of Hsiang and Kleiner \cite{HsiangKleiner} and the classification of positively curved homogeneous spaces that was resolved in the 1970's; see \cite{Wallach_72,BerardBergery,WZHomog}.
	The Grove symmetry program has resulted in many major classification results while also motivating constructions of new examples of manifolds with lower curvature bounds along with discoveries of unexpected topological structure of positively curved manifolds.
	For more information, see \cite{geomofandviasym,KennardWiemelerWilking} and the references therein.
	One common notion of ``large isometry group'' is large symmetry rank.
	
	\begin{definition}
		The \textit{symmetry rank} of a Riemannian manifold $(M,g)$ is the rank of its isometry group.
		We denote this quantity by $\symrank(M,g)$. 
		Equivalently, $\symrank(M,g)$ is the maximal dimension of a torus that can act effectively and by isometries on $(M,g)$.
	\end{definition} 
	
	Grove and Searle proved that any closed, connected, $n$-dimensional manifold with positive sectional curvature must have $\symrank(M,g)\leq\lfloor\frac{n+1}{2}\rfloor$, and in the case of equality, the manifold must be diffeomorphic to a sphere, complex or real projective space, or a lens space \cite{maxsymrank}.
	Since then, many have studied closed manifolds with positive sectional curvature and large symmetry rank; see, for example, \cite{connprinc}, \cite{Rong_2002}, \cite{Fang_2005}, \cite{KennardLogBounds}, \cite{AmannKennardDim10-16}, \cite{KennardWiemelerWilking}. 	
	Motivated by these works and, more generally, the success of the Grove symmetry program, we study manifolds with positive intermediate Ricci curvature and large symmetry rank in this article.
	
	\begin{definition}
		A Riemannian manifold $(M,g)$ is said to have \textit{positive $\mathit{2}^\mathit{nd}$-intermediate Ricci curvature} if, for any choice of orthonormal vectors $\{u,e_1,e_2\}$, the sum of sectional curvatures $\sec(u,e_1)+\sec(u,e_2)$ is positive.
		We abbreviate this condition as $\Ric_2(M,g)>0$, omitting $M$ or $g$ when they are understood.
	\end{definition}
	
	For our first main result, we generalize the Grove-Searle maximal symmetry rank theorem from \cite{maxsymrank} to the setting of manifolds with $\Ric_2>0$:
	
	\begin{result}\label{result:main}
		Let $(M^n,g)$ be a closed, connected Riemannian manifold of dimension $n\geq 3$ with $\Ric_2>0$.
		Then 
		\[
			\symrank(M^n,g) \leq \left\lfloor \frac{n+1}{2} \right\rfloor.
		\]
		Furthermore, if $M^n$ is simply connected and a torus $T^r$ of rank $r = \lfloor \frac{n+1}{2} \rfloor$ acts effectively and by isometries on $M^n$, then the following hold:
		\begin{enumerate}
			\item If the dimension $n$ is odd, then $M^n$ is diffeomorphic to a sphere.
			
			\item If the dimension $n$ is even, $n\neq 6$, and the second Betti number satisfies $b_2(M^n) \leq 1$, then $M^n$ is either diffeomorphic to a sphere or homeomorphic to a complex projective space.
		\end{enumerate} 
	\end{result}
	
	The dimension $n=4$ case in Theorem \ref{result:main} follows from purely topological considerations; see Corollary \ref{cor:dim4} below.
	For even dimensions $n\geq 8$, the tools developed below (Proposition \ref{result:torusfixedpt} and Theorem \ref{result:connprinc}) allow us to show that the odd degree integral cohomology groups are trivial (see Theorem \ref{result:even}), and the even degree cohomology is periodic, in the sense that $H^2(M^n;\mathbb{Z}) \cong H^4(M^n;\mathbb{Z}) \cong H^6(M^n;\mathbb{Z}) \cong \cdots \cong H^{n-2}(M^n;\mathbb{Z})$.
	This is where the assumption that $b_2(M^n)\leq 1$ allows us to establish a classification.
	In dimension $6$, these tools do not guarantee similar restrictions on the cohomology groups, which is why the $n=6$ case is excluded in Theorem \ref{result:main}.
	Furthermore, $S^3\times S^3$ admits a metric with $\Ric_2>0$ and maximal symmetry rank; see Example \ref{ex:S3xS3} for more information.
	Thus, we cannot hope to have the same conclusion as Theorem \ref{result:main} for a classification of $6$-dimensional manifolds with $\Ric_2>0$ and maximal symmetry rank.
	
	We note that Alexandrov geometry is a commonly used tool when studying isometric group actions on manifolds with positive sectional curvature. 
	For example, Grove and Searle use the concavity of the distance functions from the boundaries of positively curved orbit spaces to establish their diffeomorphism classification in \cite{maxsymrank}.
	However, because sectional curvatures are allowed to be negative when $\Ric_2>0$, Alexandrov geometry is far less useful for our purposes.
	Instead, we rely on a generalization of Wilking's connectedness principle (Theorem \ref{result:connprinc} below) and topological results by Montgomery and Yang \cite{MontgomeryYang} and Fang and Rong \cite{FangRong2004}.

	
%
	
	If we weaken the symmetry rank assumption in Theorem \ref{result:main}, we show that we can still obtain classifications.
	Specifically, in odd dimensions, we prove the following:

	\begin{result}\label{result:odd}
		Let $M^n$ be a closed, simply connected Riemannian manifold of odd dimension $n\geq 7$ with $\Ric_2>0$.
		Suppose a torus $T^r$ of rank $r\geq\tfrac{3n+10}{8}$ acts  effectively and by isometries on $M^n$.
		Then $M^n$ is homeomorphic to a sphere.
	\end{result}
	
	In even dimensions, we also establish the following:

	\begin{result}\label{result:even}
		Let $M^n$ be a closed, simply connected, Riemannian manifold of even dimension $n\geq 8$ with $\Ric_2>0$.
		Suppose a torus $T^r$ of rank $r\geq\tfrac{3n+6}{8}$ acts effectively and by isometries on $M^n$.
		Then $H^i(M^n;\mathbb{Z})=0$ for all odd values of $i$.
		Furthermore, if the second Betti number satisfies $b_2(M^n)\leq1$, then $M^n$ is either homeomorphic to a sphere or tangentially homotopy equivalent to a complex projective space.
	\end{result}
	
	Recall that manifolds $M$ and $N$ are said to be tangentially homotopy equivalent if there is a homotopy equivalence $h:M\to N$ such that the pullback bundle $h^*TN$ is stably isomorphic to the tangent bundle $TM$.
	Equivalently, $M$ is tangentially homotopy equivalent to $N$ if there exists $l\geq 1$ such that $M\times \mathbb{R}^l$ is diffeomorphic to $N\times \mathbb{R}^l$.
	We use a result of Dessai and Wilking \cite{DessaiWilking} to establish the complex projective space case of the classification in Theorem \ref{result:even}.
	
	
	Hopf famously conjectured that closed, even-dimensional manifolds with positive sectional curvature must have positive Euler characteristic. 
	In an effort to make progress toward this conjecture, many have shown that even-dimensional manifolds with positive sectional curvature and large symmetry rank must have positive Euler characteristic; see, for example, \cite{puttmannsearle}, \cite{Rong_2002}, \cite{hopfconjwithsymm}, \cite{KennardWiemelerWilking}.
	We establish a similar result for manifolds with $\Ric_2>0$:

	\begin{result}\label{result:evenEuler}
		Suppose $M^n$ is a closed Riemannian manifold of even dimension $n\geq 8$ with $\Ric_2>0$.  
		If a torus $T^r$ of rank $r\geq \frac{n}{4}+2$ acts effectively and by isometries on $M^n$, then $\chi(M^n)>0$.
	\end{result}
	
	The tools we develop to establish the results above should prove to be helpful in studying general isometric group actions on closed manifolds with positive $k^\mathrm{th}$-intermediate Ricci curvature for any $k\geq 2$.
	
	\begin{definition}
		We say an $n$-dimensional Riemannian manifold $(M^n,g)$ has \textit{positive $k^\text{th}$-intermediate Ricci curvature} for $k\in\{1,\dots,n-1\}$ if, for any choice of orthonormal vectors $\{u,e_1,\dots,e_k\}$, the sum of sectional curvatures $\sum_{i=1}^k \sec(u,e_i)$ is positive. \footnote{Positive $k^\mathrm{th}$-intermediate Ricci curvature should not be confused with $k$-positive Ricci curvature; see, for example, \cite{CrowleyWraith} and the references therein.}
		We abbreviate this condition by writing $\Ric_k(M^n,g)>0$, omitting $M^n$ or $g$ when they are understood.
	\end{definition}
	
	Notice that $\Ric_1>0$ is equivalent to positive sectional curvature, and $\Ric_{n-1}>0$ is equivalent to positive Ricci curvature.
	Furthermore, if $\Ric_k>0$, then $\Ric_l>0$ for all $l\geq k$. 
	Thus $\Ric_2>0$ is a strong condition on curvature in this hierarchy, second only to positive sectional curvature.
	
	See Section \ref{sec:examples} for basic examples of manifolds with $\Ric_k>0$.
	For the classification of compact symmetric spaces by the minimal value of $k$ for which each has $\Ric_k>0$, see \cite{AmannQuastZarei} or \cite{DavidLawrenceMiguel}.
	We note that the only compact irreducible symmetric spaces that have $\Ric_2>0$ are those of rank $1$, i.e. those that have positive sectional curvature.
	For more constructions of $\Ric_k>0$, see Theorems E, F, and G in \cite{DavidLawrenceMiguel}.
	
	The first tool we establish in this article guarantees fixed point sets of torus actions on closed manifolds with $\Ric_k>0$:

	
	\begin{propositionresult}\label{result:torusfixedpt}
		Suppose $M^n$ is a closed $n$-dimensional Riemannian manifold with $\Ric_k>0$ for some $k\in\{2,\dots,n-1\}$. 
		If a torus $T^r$ of rank $r\geq k+1$ acts by isometries on $M^n$, then there is a codimension-$k$ torus subgroup $T^{r-k}\subset T^r$ such that the $T^{r-k}$-action on $M^n$ has a fixed point.
	\end{propositionresult}
	
	Berger proved in \cite{BergerKilling} that any Killing field on a closed, even-dimensional manifold with positive sectional curvature must have a zero. 
	It follows that an isometric circle action on such a manifold must have a fixed point. 
	Sugahara in \cite{sugahara} extended Berger's argument to prove that on a closed manifold of any dimension with positive sectional curvature, any two commuting Killing fields  must be linearly dependent at some point. 
	It follows that any isometric torus action on such a manifold must have a circle sub-action with non-empty fixed point set, a result that was established independently by Grove and Searle in \cite{maxsymrank}. 
	Grove refers to these collective results as the Isotropy Rank Lemma in \cite{geomofandviasym}.
	To prove Proposition \ref{result:torusfixedpt}, we use a generalization of Sugahara's approach; see Proposition \ref{prop:lindep} below.

%
	
	As mentioned above, Grove and Searle proved in \cite{maxsymrank} that any closed, connected $n$-dimensional manifold with positive sectional curvature cannot have symmetry rank greater than $\lfloor\frac{n+1}{2}\rfloor$.
	The author proved in \cite{localsymrank} that any connected $n$-manifold with $\Ric_k>0$ at a point cannot have symmetry rank greater than $\lfloor\frac{n+k}{2}\rfloor$, without assuming that the manifold is closed. 
	For our second new tool, we apply Proposition \ref{result:torusfixedpt} to obtain the following refined symmetry rank bounds for closed manifolds with $\Ric_k>0$:
	
	\begin{propositionresult}\label{result:symrankbound}
		Suppose $(M^n,g)$ is a closed, connected, $n$-dimensional Riemannian manifold with $\Ric_k>0$ for some $k\in\{3,\dots,n-1\}$. 
		Then 
		\[
			\symrank(M^n,g)\leq \left\lfloor\frac{n+k}{2}\right\rfloor-1.
		\]
	\end{propositionresult}
	
	It follows from Theorem \ref{result:main} and Proposition \ref{result:symrankbound} that if $k\leq 3$, or if $k=4$ and $n$ is odd, then the symmetry rank of $M^n$ is at most $\left\lfloor\frac{n+1}{2}\right\rfloor$, the same bound as for manifolds with positive sectional curvature.
	In particular, the symmetry rank bounds are optimal in these cases; see Examples \ref{ex:GS}, \ref{ex:prod}, and \ref{ex:S3xS3symrank} for more information.
	It also follows from these examples that the symmetry rank bounds are optimal in dimension $n=6$ for all values of $k$.
	
	
	Recall that for an $n$-dimensional manifold, $\Ric_{n-1}>0$ is equivalent to positive Ricci curvature.
	It follows from work of Pak \cite{pak} and Parker \cite{parker} that closed manifolds of dimension $n\geq 4$ cannot have positive Ricci curvature while having symmetry rank $\geq n-1$; see Remark \ref{rmk:PakParker} below.
	Thus, the symmetry rank for closed manifolds of dimension $n\geq 4$ with positive Ricci curvature is bounded above $n-2$.
	This fact is reflected in Proposition \ref{result:symrankbound} in the case $k=n-1$.
	In dimensions $n\leq 6$, spheres or their Riemannian products provide examples of Ricci positive manifolds with symmetry rank $n-2$.
	Corro and Galaz-Garc\'ia show in \cite{riccisymrank} that for dimensions $n\geq 7$, there exist closed, simply connected $n$-dimensional manifolds that admit metrics of positive Ricci curvature with symmetry rank $n-4$. 
	It is still unknown whether it is possible to find $n$-dimensional manifolds that admit metrics of positive Ricci curvature with symmetry rank $n-2$ for $n\geq 7$.
	
	Our last tool generalizes Wilking's Connectedness Principle from \cite{connprinc} for fixed point sets of isometric group actions on closed manifolds with $\Ric_k>0$: 
	
	\begin{result}\label{result:connprinc}
		Let $M^n$ be a compact, $n$-dimensional manifold with $\Ric_k>0$ for some $k\in\{2,\dots,n-1\}$. 
		Suppose $N^{n-d}$ is a compact embedded submanifold of codimension $d$ in $M^n$. If there is a Lie group $G$ that acts by isometries on $M^n$ and fixes $N^{n-d}$ point-wise, then the inclusion 
		\[
			N^{n-d}\hookrightarrow M^n\text{ is }(n-2d+2-k+\delta(G))\text{-connected},
		\] 
		where $\delta(G)$ is the dimension of the principal $G$-orbits in $M^n$.
	\end{result}
	
	Recall that a map $f:N\to M$ is \textit{$j$-connected} if the induced map $\pi_i(f):\pi_i(N)\to\pi_i(M)$ is an isomorphism for $i<j$ and an epimorphism for $i=j$. 
	The proof of Theorem \ref{result:connprinc} relies on a Morse-theoretic argument on the space of curves in $M$ that start and end in $N$.
	Wilking's proof of the positive sectional curvature ($k=1$) case does not extend to the setting of $\Ric_k>0$ for $k\geq 2$.
	This is because he uses a Cheeger deformation argument that relies on the assumption of positive sectional curvature.
	However, Wilking does observe in \cite[Remark 2.4]{connprinc} that for a totally geodesic submanifold $N^{n-d}$ of a closed manifold $M^n$ with $\Ric_k(M^n)>0$, the inclusion $N^{n-d}\hookrightarrow M^n$ is $(n-2d+2-k)$-connected, without assuming existence of an isometric group action. 
	For a generalization of Wilking's result to submanifolds that are not necessarily totally geodesic in manifolds with $\Ric_k>0$, see the work of Fang, Mendon\c{c}a, and Rong in \cite{FangMendoncaRong}.
	Guijarro and Wilhelm established a quantitative version of Wilking's Connectedness Principle for manifolds with positive intermediate Ricci curvature in \cite{softconnprinc} using a Jacobi field comparison result.
	Again, they do not assume existence of an isometric group action.
	Our argument for Theorem \ref{result:connprinc} is adapted from Guijarro and Wilhlelm's approach.
	
	\subsection{Organization}
		
		In Section \ref{sec:examples}, we present examples of closed manifolds with $\Ric_k>0$ and discuss their symmetry rank.
		In Section \ref{sec:fixedpointsets}, we study zero sets of commuting Killing fields on closed manifolds with $\Ric_k>0$, and we prove Proposition \ref{result:torusfixedpt}. 
		In Section \ref{sec:symrankbound}, we use Proposition \ref{result:torusfixedpt} to prove our symmetry rank bounds in Theorem \ref{result:main} and Proposition \ref{result:symrankbound}.
		In Section \ref{sec:connprinc}, we establish our connectedness principle, Theorem \ref{result:connprinc}.
		Finally, in Sections \ref{sec:odd} and \ref{sec:even}, we study closed manifolds with $\Ric_2>0$ and large symmetry rank, focusing on odd dimensions in Section \ref{sec:odd} and even dimensions in Section \ref{sec:even}.
		We prove the following results in the corresponding subsections: 
		the odd dimensional case of Theorem \ref{result:main} in Section \ref{sec:odddiffeo}, Theorem \ref{result:odd} in Section \ref{sec:oddhomeo}, Theorem \ref{result:evenEuler} in Section \ref{sec:posEuler}, the even dimensional case of Theorem \ref{result:main} in Section \ref{sec:evenmaxsymrank}, and Theorem \ref{result:even} in Section \ref{sec:even3/4}.
		
	\subsection{Acknowledgements} 
		
		Some of the results in this article are part of my doctoral thesis at the University of California, Riverside.
		I thank my advisor, Fred Wilhelm, for his valuable guidance and feedback.
		Other parts of this article were completed while I was employed as a part-time postdoctoral researcher by the Research Foundation of the City University of New York.
		I thank Christina Sormani for that opportunity.
		I acknowledge Luis Guijarro and Fred Wilhelm for their consultation on the proof of Theorem \ref{result:connprinc}.
		I also thank Manuel Amann for suggestions on improving Theorem \ref{result:even}.
		I thank Fernando Galaz-Garc\'ia for several helpful suggestions and for bringing the article \cite{MontgomeryYang} to my attention.
		I also thank the referee for many helpful suggestions which helped me to improve the exposition of this article.
		Finally, I thank Lee Kennard for many stimulating discussions concerning this project and for his guidance in the preparation of this article.
		I thank Syracuse University for hosting my visit to collaborate with Lee.
	
	\section{Examples}\label{sec:examples}

In this section, we present a few examples of closed manifolds with $\Ric_k>0$ and large torus actions.
First, we recall maximal torus actions for manifolds with positive sectional curvature:

\begin{example}\label{ex:grovesearle}
	Grove and Searle proved in \cite{maxsymrank} that any closed, connected, $n$-dimensional Riemannian manifold $(M,g)$ with positive sectional curvature must have $\symrank(M,g)\leq\lfloor\frac{n+1}{2}\rfloor$, and in the case of equality, the manifold must be diffeomorphic to a sphere, complex or real projective space, or a lens space. 
	Here, $\lfloor x\rfloor$ denotes the largest integer less than or equal to the quantity $x$.
	A maximal torus action $T^m\times S^{2m-1}\to S^{2m-1}$ on the odd-dimensional unit sphere $S^{2m-1}\subset\mathbb{C}^m$ is given in complex coordinates by 
	\[
		(e^{i\theta_1},\dots,e^{i\theta_m}) \cdot (z_1,\dots,z_m) \defeq (e^{i\theta_1}z_1,\dots,e^{i\theta_m}z_m).
	\]
	This action induces an effective $T^m$-action on all lens spaces $S^{2m-1}/\mathbb{Z}_q$.
	Furthermore, letting $\Delta S^1$ denote the diagonal embedding of the circle in $T^m$, this $T^m$-action on $S^{2m-1}$ descends to an effective action by $T^{m-1}\cong T^m/\Delta S^1$ on $\mathbb{C}\mathrm{P}^{m-1} = S^{2m-1}/\Delta S^1$.
	Finally, a maximal $T^m$-action on the even-dimensional sphere $S^{2m}\subset\mathbb{C}^m\oplus\mathbb{R}$ is given by suspending the action above, i.e. with the action on the $\mathbb{R}$-factor being trivial.
	This action on $S^{2m}$ induces an effective $T^m$-action on $\mathbb{R}\mathrm{P}^{2m}$.
\end{example}

Next, we discuss the elementary class of examples given by Riemannian products:

\begin{example}\label{ex:product}
	Consider a collection of Riemannian manifolds $\{(M_i,g_i)\}_{i=1}^N$, each of dimension at least $2$ and with $\Ric_{k_i}(M_i,g_i)>0$ for some $k_i\in\{1,\dots,\dim M_i-1\}$.
	It is a straightforward exercise to check that, with respect to the product metric, 
	\begin{equation}\label{eq:product}
		\Ric_k(M_1\times \dots\times M_N)>0\;\text{ only for }\;k\geq\max_{i=1,\dots,N}\left\{k_i+{\textstyle\sum}_{j\neq i}\dim M_j\right\}.
	\end{equation}
	For example, consider spheres $S^n$ and $S^m$ of dimension at least $2$ with the standard round metrics. 
	Then with respect to the product metric, $S^n\times S^m$ has $\Ric_k>0$ only for $k\geq\max\{n+1,m+1\}$ and has symmetry rank equal to $\lfloor\frac{n+1}{2}\rfloor+\lfloor\frac{m+1}{2}\rfloor$.
	
\end{example}

Notice it follows from \eqref{eq:product} that no Riemannian product will have $\Ric_2>0$.
Finally, we present a metric on $S^3\times S^3$ that has $\Ric_2>0$ and maximal symmetry rank:

\begin{example}\label{ex:S3xS3}
	Consider the sphere $S^3\subset \mathbb{H}$ as the Lie group of unit quaternions.
	Let $\Delta S^3$ denote the diagonal embedding of $S^3$ as a subgroup of the product $S^3\times S^3\times S^3$.
	Suppose $S^3\times S^3\times S^3$ is equipped with the standard Riemannian product metric, $g_\mathrm{prod}$, where each factor of $S^3$ has the standard (round) biinvariant metric.
	Then the $\Delta S^3$-action on $S^3\times S^3\times S^3$ by right multiplication is free and by isometries, and hence the quotient $(S^3\times S^3\times S^3)/\Delta S^3$ inherits a Riemannian metric $g_\mathrm{quot}$ such that the quotient map $(S^3\times S^3\times S^3,g_\mathrm{prod}) \to ((S^3\times S^3\times S^3)/\Delta S^3,g_\mathrm{quot})$ is a Riemannian submersion.
	Now, the quotient is diffeomorphic to $S^3\times S^3$ via the map $(S^3\times S^3\times S^3)/\Delta S^3\to S^3\times S^3$ given by $(a,b,c)\Delta S^3 \mapsto (ac^{-1},bc^{-1})$.
	Let $g_*$ denote the pushforward of the metric $g_\mathrm{quot}$ through this map.
	Then the metric $g_*$ on $S^3\times S^3$ is left-invariant, it is invariant under the diagonal action of $S^3$ by right multiplication, and $\Ric_2(S^3\times S^3,g_*)>0$.
	For more information about this construction, including a generalization to products of compact semisimple Lie groups, see Theorem E in \cite{DavidLawrenceMiguel}.
	It follows that $\symrank(S^3\times S^3,g_*) = 3$, which is maximal for closed $6$-dimensional manifolds with $\Ric_2>0$ by Theorem \ref{result:main}.
	In particular, a maximal isometric torus action $T^3\times (S^3\times S^3)\to S^3\times S^3$ with respect to $g_*$ is given in quaternionic coordinates by 
	\[
		(p,q,r)\cdot(a,b) \defeq (par^{-1},qbr^{-1}).
	\]
	By taking quotients of free isometric circle actions, it also follows that $S^3\times S^2$ inherits a metric with $\Ric_2>0$ and symmetry rank $2$, and $S^2\times S^2$ inherits a metric with $\Ric_2>0$ and symmetry rank $1$.

\end{example}

Example \ref{ex:S3xS3} shows that the classification result by Grove and Searle in \cite{maxsymrank} fails to generalize to $\Ric_2>0$ in dimension $6$.
Also, the induced metric on $S^3\times S^2$ shows that the main result by Rong in \cite{Rong_2002} fails to generalize to $\Ric_2>0$.
Furthermore, the induced metric on $S^2\times S^2$ shows that the Hsiang-Kleiner theorem from \cite{HsiangKleiner} fails to generalize to $\Ric_2>0$.


	\section{Fixed point sets of torus actions}\label{sec:fixedpointsets}
	
	In this section, we establish Proposition \ref{result:torusfixedpt}, which asserts that a closed manifold with $\Ric_k>0$ and a large isometric torus action must have points with non-trivial isotropy. Because the action fields induced by isometric torus actions are commuting Killing fields, we will prove Proposition \ref{result:torusfixedpt} by first proving the following:
	
	\begin{proposition}\label{prop:lindep}
		Suppose $M^n$ is a closed $n$-manifold and $\Ric_k(M^n,g)>0$ for some $k\in\{1,\dots,n-1\}$. If there are $k+1$ commuting Killing fields on $M^n$, then they must be linearly dependent at some point in $M^n$.
	\end{proposition}
	
	As discussed in the introduction, Proposition \ref{prop:lindep} generalizes Sugahara's result from \cite{sugahara} stating that any two commuting Killing fields on a closed manifold with positive sectional curvature must be linearly dependent at some point; see also the proof of Theorem 8.3.5 in \cite{petersen}.

	We now setup notation for the proof of Proposition \ref{prop:lindep}. Given vector fields $Y_1,\dots,Y_k$ on $M$ that are linearly independent in each tangent space, define the distribution
	\[
		\mathcal{Y}\defeq\mathrm{span}\{Y_1,\dots,Y_k\}.
	\]
	Given a Riemannian metric on $M$ and a vector field $X$ on $M$, define the vector field $X^\perp$ at each point $p$ to be the projection of $X|_p$ onto the orthogonal complement of $\mathcal{Y}|_p$, and define the function $f^\perp:M\to \mathbb[0,\infty)$ by
	\[
		f^\perp \defeq \tfrac{1}{2} \left| X^\perp \right|^2. 
	\]
	We relate the Hessian of $f^\perp$ to the curvature tensor as follows:
	\begin{lemma}\label{lem:hess}
		Let $X,Y_1,\dots,Y_k$ be commuting Killing fields on $M$ that are linearly independent in each tangent space. Suppose there is a point $p\in M$ at which $Y_1|_p,\dots,Y_k|_p$ are orthonormal and $X|_p$ is orthogonal to the subspace $\mathcal{Y}|_p\subseteq T_pM$. Then for all $v\in T_pM$, the Hessian of the function $f^\perp \defeq \frac{1}{2} | X^\perp |^2$ at the point $p$ is given by
		\[
			\Hess f^\perp(v,v)=|\nabla_vX|^2-R(v,X,X,v)-4\sum_{i=1}^k\langle\nabla_vX,Y_i\rangle^2.
		\]
	\end{lemma}
	
	\begin{proof}
		Let $X^\top$ denote the projection of $X$ onto $\mathcal{Y}$, and define the functions $f:M\to \mathbb[0,\infty)$ and $f^\top:M\to \mathbb[0,\infty)$ by
		\[
			f \defeq \tfrac{1}{2} \left| X \right|^2, \quad \text{and} \quad f^\top \defeq \tfrac{1}{2} \left| X^\top \right|^2.
		\]
		Then $f^\perp = f - f^\top$, and because $X$ is a Killing field, 
		\[
			\Hess f (v,v) = |\nabla_v X|^2 - R(v,X,X,v).
		\]
		Thus, it suffices to show that at the point $p$, we have $\Hess f^\top(v,v) = 4 \sum_{i=1}^k \langle \nabla_v X , Y_i \rangle^2$.
		Because $Y_1,\dots,Y_k$ are linearly independent in each tangent space, we may perform a Gram-Schmidt process to produce the following (not necessarily Killing) orthonormal vector fields on $M$:
		\[
			\hat{Y}_i \defeq \bar{Y}_i / |\bar{Y}_i|, \quad \text{where} \quad \bar{Y}_i \defeq Y_i - \sum_{j<i} \langle Y_i, \hat{Y}_j \rangle \hat{Y}_j.
		\]
		Then it follows that 
		\[
			f^\top = \tfrac{1}{2} \sum_{i=1}^k \langle X, \hat{Y}_i \rangle^2, \quad \text{and} \quad \grad f^\top = \sum_{i=1}^k \langle X , \hat{Y}_i \rangle \grad \langle X , \hat{Y}_i \rangle.
		\]
		Thus, the Hessian of $f^\top$ is given by
		\[
			\Hess f^\top (v,v) = \sum_{i=1}^k \left[ (D_v \langle X , \hat{Y}_i \rangle)^2 + \langle X , \hat{Y}_i \rangle \Hess \left( \langle X , \hat{Y}_i \rangle \right) (v,v) \right].
		\]
		Now at the point $p$, because $Y_1|_p,\dots,Y_k|_p$ are orthonormal and $X|_p$ is orthogonal to $\mathcal{Y}|_p$, it follows that $\hat{Y}_i|_p = Y_i|_p$, $\langle X, \hat{Y}_i \rangle |_p = 0$, and $D_v \langle X , \hat{Y}_i \rangle|_p = D_v \langle X , Y_i \rangle|_p$ for all $i$.
		Furthermore, because $X,Y_1,\dots,Y_k$ are commuting Killing fields, $D_v \langle X , Y_i \rangle = 2 \langle \nabla_v X , Y_i \rangle$ for all $i$.
		Therefore, the result follows.
	\end{proof}
	
	We can now use Lemma \ref{lem:hess} to prove Proposition \ref{prop:lindep}:
	
	\begin{proof}[Proof of Proposition \ref{prop:lindep}]
		We will prove the contrapositive of Proposition \ref{prop:lindep}.
		Suppose $X,Y_1,\dots,Y_k$ are $k+1$ commuting Killing fields on $M$ that are linearly independent in each tangent space. 
		We will show there exists a point $p\in M$ and orthonormal vectors $u,e_1,\dots,e_k\in T_pM$ such that $\sum_{i=1}^k \sec(u,e_i)\leq 0$. 
		
		Because $X,Y_1,\dots,Y_k$ are linearly independent, $f^\perp \defeq \frac{1}{2}|X^\perp|^2$ must attain a positive minimum at some point $p$. By performing a Gram-Schmidt process at $p$, we can find commuting Killing fields $Y_1,\dots,Y_k$ that span the same distribution $\mathcal{Y}$ and are orthonormal at $p$, and doing so does not change the values of $f^\perp $. Similarly, we can replace $X$ with the Killing field that commutes with $Y_1,\dots,Y_k$ such that $X|_p$ is orthogonal to $\mathcal{Y}|_p$, and this too will not change the values of $f^\perp $. 
		
		Now, with these new choices of $X,Y_1,\dots,Y_k$, by Lemma \ref{lem:hess}, we know for $v\in T_pM$,
		\begin{equation}\label{eq:Hess}
			\Hess f^\perp (v,v)=|\nabla_vX|^2-R(v,X,X,v)-4\sum_{i=1}^k\langle\nabla_vX,Y_i\rangle^2.
		\end{equation}
		Let $P^\perp$ denote the projection onto the orthogonal complement of $\mathcal{Y}|_p\subseteq T_pM$.
		Then at $p$, we have
		\begin{align*}
			| \nabla_v X |^2 &= \langle \nabla_v X , P^\perp ( \nabla_v X ) \rangle + \sum_{i=1}^k \big\langle \nabla_v X , \langle \nabla_v X , Y_i \rangle Y_i \big\rangle\\
			&= | P^\perp ( \nabla_v X ) |^2 + \sum_{i=1}^k \langle \nabla_v X , Y_i \rangle^2.
		\end{align*}
		Also, because $f^\perp $ attains a minimum at $p$, we have $\Hess f^\perp (v,v)\geq 0$.
		Thus by Equation \eqref{eq:Hess}, for all $v\in T_pM$,
		\begin{equation}\label{eq:curv}
			R(v,X,X,v)\leq | P^\perp ( \nabla_v X ) |^2 - 3 \sum_{i=1}^k\langle\nabla_vX,Y_i\rangle^2.
		\end{equation}
		Hence, it suffices to show that $\ker ( P^\perp \circ \nabla X)$ is at least $k+1$ dimensional, where $\nabla X : T_p M \to T_p M$ denotes the linear map given by $v \mapsto \nabla_v X$.
		First, note that
		\begin{align}\label{eq:kernel}
			\dim ( \ker ( P^\perp \circ \nabla X) ) &= \dim ( \ker \nabla X ) + \dim ( \Image \nabla X \cap \ker P^\perp ) \nonumber\\
			&= \dim ( \ker \nabla X ) + \dim ( \Image \nabla X \cap \mathcal{Y} )\nonumber\\
			&= \dim ( \ker \nabla X ) + \dim ( \Image \nabla X ) + \dim ( \mathcal{Y} ) - \dim ( \Image \nabla X + \mathcal{Y} )\nonumber\\
			&= n + k - \dim ( \Image \nabla X + \mathcal{Y} ).
		\end{align}
		Now recall that, because $X$ is a Killing field, $\nabla X$ is a skew-symmetric linear map. 
		In particular, given $u\in \ker \nabla X$ and $v\in T_pM$, 
		\[
			\langle \nabla_v X , u \rangle = - \langle \nabla_u X , v \rangle = 0.
		\]
		Thus $\Image \nabla X \subseteq ( \ker \nabla X )^\perp$.
		Finally, we show that $X\in\ker\nabla X$. 
		With our choices of $X,Y_1,\dots,Y_k$, we have 
		\[
			f^\perp = \tfrac{1}{2}|X^\perp|^2\leq\tfrac{1}{2}|X|^2, 
		\]
		with equality at $p$. 
		So defining $f \defeq\frac{1}{2}|X|^2$, we know $f $ also attains a minimum at $p$. 
		Hence $\grad f =-\nabla_XX=0$ at $p$, and thus $X\in \ker\nabla X$. 
		Now because $X$ is not an element of $( \ker \nabla X )^\perp$ or $\mathcal{Y}$ at $p$, we have 
		\[
			\dim(\Image \nabla X + \mathcal{Y} ) \leq \dim( ( \ker \nabla X )^\perp + \mathcal{Y} ) \leq n - 1.
		\]
		Thus, applying this inequality to Equation \eqref{eq:kernel}, we have that $\ker ( P^\perp \circ \nabla X)$ is at least $k+1$ dimensional.
		Hence, we can choose orthonormal vectors $v_1,\dots, v_k$ so that $X,v_1,\dots, v_k$ is an orthogonal basis of $\ker ( P^\perp \circ \nabla X)$, and by Inequality \eqref{eq:curv}, we have $\sum_{i=1}^k R(v_i,X,X,v_i)\leq 0$.
		Therefore, we have proven Proposition \ref{prop:lindep} by contraposition.
	\end{proof}
	
	We now use Proposition \ref{prop:lindep} to prove Proposition \ref{result:torusfixedpt}, which we restate here for convenience:
	
	\begin{proposition}\label{prop:torusfixedpt}
		Suppose $M^n$ is a closed $n$-dimensional Riemannian manifold with $\Ric_k>0$ for some $k\in\{2,\dots,n-1\}$. 
		If a torus $T^r$ of rank $r\geq k+1$ acts by isometries on $M^n$, then there is a codimension-$k$ torus subgroup $T^{r-k}\subset T^r$ such that the $T^{r-k}$-action on $M^n$ has a fixed point.
	\end{proposition}
	
	\begin{proof}
		Fix $k\geq 2$. We will prove Proposition \ref{prop:torusfixedpt} by induction on $r$.
		Given a natural number $r$, let $\mathfrak{t}^r$ denote the Lie algebra for the torus $T^r$. 
		Given a $T^r$-action an a manifold $M$ and a point $p\in M$, let $K_p:\mathfrak{t}^r\to T_pM$ denote the linear map such that for any $z\in \mathfrak{t}^r$, $K_p(z)$ is the evaluation at $p$ of Killing field induced by $z$ on $M$ via the $T^r$-action.  
		
		For the base case, $r=k+1$, suppose that $T^{k+1}$ acts isometrically on $(M^n,g)$. Choose a basis $x_1,\dots,x_{k+1}$ for $\mathfrak{t}^{k+1}$. Then the action fields $X_1,\dots,X_{k+1}$ are commuting Killing fields on $M^n$. So by Proposition \ref{prop:lindep}, the fields $X_1,\dots,X_{k+1}$ must be linearly dependent at some point $p\in M^n$. Thus, the kernel of the linear map $K_p:\mathfrak{t}^{k+1}\to T_pM^n$ must be at least $1$-dimensional.
		Thus, there is a circle subgroup $T^1\subset T^{k+1}$ that fixes the point $p$.
		
		For the sake of induction, suppose there exists $r_0\geq k+2$ such that the statement of Proposition \ref{prop:torusfixedpt} holds for all $r\leq r_0-1$. We will now show the same conclusion holds for $r=r_0$.
		
		Suppose $T^{r_0}$ acts isometrically on $(M^n,g)$. Choose a linearly independent set of vectors $x_1,\dots,x_{k+1}\in\mathfrak{t}^{r_0}$. By Proposition \ref{prop:lindep}, the action fields $X_1,\dots,X_{k+1}$ must be linearly dependent at some point $p\in M^n$. 
		As before, the kernel of the map $K_p:\mathfrak{t}^{r_0}\to T_pM^n$ is at least $1$-dimensional, and we have a circle subgroup $T^1\subset T^{r_0}$ that fixes $p$. 
		Let $N$ be a connected component of this fixed point set for the $T^1$-action on $M^n$. Then $N$ is totally geodesic in $(M^n,g)$ and is invariant under the $T^{r_0}$-action. If $\dim(N)\leq k$, then because $r_0>k$, the kernel of the $T^{r_0}$-action on $N$ must contain a torus subgroup of dimension at least $r_0-k$, and the result follows. If $\dim(N)\geq k+1$, then because $N$ is totally geodesic in $(M^n,g)$, we have $\Ric_k(N)>0$. Thus induction hypothesis applies to the action of $T^{r_0-1}\defeq T^{r_0}/T^1$ on $N$. So there is a codimension-$k$ subtorus $T^{r_0-1-k}\subset T^{r_0-1}$ that has a fixed point in $N$. Therefore, because $N$ is fixed by $T^1$, the torus $T^{r_0-k}\defeq T^1\times T^{r_0-1-k}$ has a fixed point in $M^n$. 
	\end{proof}

	\section{Symmetry rank bound}\label{sec:symrankbound}
	
	In this section, we prove the symmetry rank bounds from Theorem \ref{result:main} and Proposition \ref{result:symrankbound}.
	In particular, we intend to prove the following:
	
	\begin{proposition}\label{prop:symrankbound}
		Suppose $(M^n,g)$ is a closed, connected, $n$-dimensional Riemannian manifold with $\Ric_k>0$ for some $k\in\{1,\dots,n-1\}$.
		\begin{enumerate}
			\item \label{item:GS} If $k=1$, then $\symrank(M^n,g)\leq\left\lfloor\frac{n+1}{2}\right\rfloor$. \cite{maxsymrank}
			
			\item \label{item:2odd} If $n$ is odd and $k=2$, then $\symrank(M^n,g)\leq\left\lfloor\frac{n+1}{2}\right\rfloor$. \cite{localsymrank}
			
			\item \label{item:odd} If $n$ is odd and $k\geq 3$, then $\symrank(M^n,g)\leq \left\lfloor\frac{n+k}{2}\right\rfloor-1$.
			
			\item \label{item:even} If $n$ is even and $k\geq 2$, then $\symrank(M^n,g)\leq \left\lfloor\frac{n+k}{2}\right\rfloor-1$.
		\end{enumerate}
	\end{proposition}
	
	Notice if $k=2$, then the symmetry rank bound provided in Item \eqref{item:even} of Proposition \ref{prop:symrankbound} is equal to $\lfloor\frac{n+1}{2}\rfloor$.
	It follows that Proposition \ref{prop:symrankbound} is equivalent to the symmetry rank bounds stated in Theorem \ref{result:main} and Proposition \ref{result:symrankbound}.

	\begin{example}\label{ex:GS}
		Recall if a Riemannian manifold has positive sectional curvature, then it has $\Ric_k>0$ for any $k\geq 1$.
		As shown in \cite{maxsymrank}, spheres, real or complex projective spaces, and lens spaces all admit metrics with positive sectional curvature and symmetry rank equal to $\lfloor\frac{n+1}{2}\rfloor$; see Example \ref{ex:grovesearle} for more information.
		The symmetry rank upper bounds provided in Proposition \ref{prop:symrankbound} are equal to $\lfloor\frac{n+1}{2}\rfloor$ if $n$ is odd and $k\leq 4$, or if $n$ is even and $k\leq 3$. 
		Thus, the symmetry rank bound from Proposition \ref{prop:symrankbound} is optimal in these cases.
	\end{example}
	
	\begin{example}\label{ex:prod}
		Given a product of spheres $S^n\times S^m$, let $g_{n,m}$ denote the Riemannian product metric associated with the standard round metrics on each factor.
		As mentioned in Example \ref{ex:product}, the Riemannian product $(S^n\times S^m,g_{n,m})$ has $\Ric_k>0$ only for $k\geq \max\{n+1,m+1\}$ and $\symrank=\lfloor\frac{n+1}{2}\rfloor+\lfloor\frac{m+1}{2}\rfloor$.
		Thus the following have maximal symmetry rank for the appropriate dimension and value of $k$ by Proposition \ref{prop:symrankbound}:
		\begin{itemize}
			\item $\Ric_3(S^2\times S^2,g_{2,2})>0$, and $\symrank(S^2\times S^2,g_{2,2})=2$,  
			\item $\Ric_4(S^3\times S^2,g_{3,2})>0$, and $\symrank(S^3\times S^2,g_{3,2})=3$, 
			\item $\Ric_4(S^3\times S^3,g_{3,3})>0$, and $\symrank(S^3\times S^3,g_{3,3})=4$.
		\end{itemize}
		From the last example, it follows that the symmetry rank bound in Proposition \ref{prop:symrankbound} is also optimal in dimension $6$ with $k=4,5$.
	\end{example}
	
	\begin{example}\label{ex:S3xS3symrank}
		As we described in Example \ref{ex:S3xS3}, $S^3\times S^3$ admits a metric which has $\Ric_2>0$ and $\symrank=3$, which is maximal symmetry rank in dimension $6$ for $\Ric_2>0$ by Proposition \ref{prop:symrankbound}.
	\end{example}
	
	Now, we will prove Proposition \ref{prop:symrankbound} for fixed values of $k$ by using induction on the dimension $n$. For the base cases, we will use the following:
	
	\begin{lemma}\label{lem:Riccohom2}
		A closed manifold of dimension $n\geq 4$ cannot support a metric of positive Ricci curvature that is invariant under an effective action by a torus of rank $n-1$. 
	\end{lemma}
	
	\begin{remark}\label{rmk:PakParker}
		Lemma \ref{lem:Riccohom2} follows from the work of Pak in \cite{pak} and Parker in \cite{parker}, who showed that in dimensions $\geq 4$, closed manifolds which admit cohomogeneity-one torus actions must have infinite fundamental group. Thus by the Bonnet-Myers theorem, such manifolds cannot admit invariant metrics of positive Ricci curvature. 
	\end{remark}
	
	Next, we recall the following:
	
	\begin{proposition}[Proposition 8.3.8 in \cite{petersen}]\label{prop:mincodim}
		Let $M$ be compact and assume that $X$ and $Y$ are commuting Killing fields on $M$. 
		\begin{enumerate}
			\item $Y$ is tangent to the level sets of $|X|^2$ and, hence, to the zero set of $X$.
			\item If $X$ and $Y$ both vanish on a totally geodesic submanifold $N\subset M$, then some linear combination of $X$ and $Y$ vanishes on a submanifold in $M$ of dimension larger than $N$.
		\end{enumerate}
	\end{proposition}
	
	It follows from Proposition \ref{prop:mincodim} that if $N\subset M$ is a connected component of a fixed point set for an isometric $T^2$-action on $M$, then there is a circle subgroup $S^1\subset T^2$ such that the component of its fixed point set which contains $N$ has codimension $<\codim(N)$.
	
	We are now ready to prove Proposition \ref{prop:symrankbound}
	
	\begin{proof}[Proof of Proposition \ref{prop:symrankbound}]
		Part \eqref{item:GS} was established by Grove and Searle in \cite{maxsymrank}. 
		Part \eqref{item:2odd} follows from \cite{localsymrank}, in which the author shows that if a manifold has $\Ric_k>0$ at a point, then $\symrank(M^n,g)\leq\lfloor\frac{n+k}{2}\rfloor$.
		
		We will prove Parts \eqref{item:odd} and \eqref{item:even} of Proposition \ref{prop:symrankbound} using induction on the dimension $n$. First, we establish the base cases, dimensions $n=4$ and $n=5$. If a $4$-dimensional manifold $M^4$ has $\Ric_2>0$ or $\Ric_3>0$ and $T^3$ acts isometrically on $M^4$, then the kernel of the action must contain a circle subgroup by Lemma \ref{lem:Riccohom2}. Similarly, if a $5$-dimensional manifold $M^5$ has $\Ric_3>0$ or $\Ric_4>0$ and $T^4$ acts isometrically on $M^5$, then again the kernel of the action must contain a circle by Lemma~\ref{lem:Riccohom2}. 
		
		Now, for the sake of induction, suppose that for some $n\geq 6$, Proposition \ref{prop:symrankbound} holds for all dimensions $\dim(M)\in\{4,5,\dots,n-1\}$. We wish to show that it holds for $\dim(M)=n$. So suppose $M^n$ is either odd-dimensional with $\Ric_k>0$ for some $k\geq 3$, or $M^n$ is even-dimensional with $\Ric_k>0$ for some $k\geq 2$, and assume a torus $T^r$ of rank $r$ acts isometrically and effectively on $M^n$.
		We will show that $r \leq \lfloor \frac{ n + k }{ 2 } \rfloor - 1$.
		
		If $k\geq n-2$, then $\lfloor \frac{ n + k }{ 2 } \rfloor - 1 = n-2$, and $r \leq n-2$ by Lemma \ref{lem:Riccohom2}.
		
		Now assume instead that $k\leq n-3$.
		Then $\lfloor \frac{ n + k }{ 2 } \rfloor - 1 \geq k$, so if $r\leq k$, then we are done.
		So assume $r\geq k+1$.
		Then by Proposition \ref{result:torusfixedpt}, there exists a codimension $k$ torus of $T^r$ which has a nonempty fixed point set in $M^n$.
		Define 
		\[
		l = \min\{ m : \text{there exists }T^{r-m}\subseteq T^r \text{ such that } M^{T^{r-m}}\neq \varnothing\},
		\]
		and let $T^{r-l}$ be a torus subgroup that realizes this minimum.
		Then by Proposition \ref{result:torusfixedpt}, $l\leq k$.
		Let $F^f$ be a component of the fixed point set $M^{T^{r-l}}$.
		Then $T^l \defeq T^r/T^{r-l}$ acts effectively on $F^f$ such that no circle subgroup of $T^l$ has a fixed point.
		Furthermore, $T^{r-l}$ acts on the normal sphere $S^{n-f-1}$ to $F^f$, and because the codimension of $F^f$ is even and $S^{n-f-1}$ has positive sectional curvature, it follows from \cite{maxsymrank} that 
		\[
		r-l\leq \frac{n-f}{2}.
		\]
		
		First, assume $f\geq k+1$, so that $\Ric_k(F^f)>0$.
		Then because $k\geq 2$, we have $f\geq 3$.
		But in the case that $k=2$, $M^n$ is assumed to be even-dimensional, and hence, $F^f$ is also even-dimensional.
		Thus, we have $f\geq 4$ in all cases, and by the induction hypothesis, $l\leq \lfloor \frac{f+k}{2}\rfloor - 1$.
		Then it follows that $r \leq \frac{n-f}{2} + \lfloor \frac{f+k}{2}\rfloor - 1 = \lfloor \frac{ n + k }{ 2 } \rfloor - 1$.
		
		Now assume instead that $f\leq k$.
		Then because no circle subgroup of $T^l$ has a fixed point, we have $l\leq f$.
		Thus $r\leq \frac{n-f}{2} + f \leq \lfloor\frac{n+k}{2}\rfloor$.
		What remains to show is that the equality case cannot occur.
		This case is equivalent to having $r-l = \frac{n-f}{2}$, $l=f$, and $f\geq k-1$.
		In this case, we may choose a subgroup $T^{r-l-1}\subset T^{r-l}$ with a fixed point component $F'$ containing $F$ such that $\dim(F') = f+2$.
		Thus $k+1 \leq \dim(F') \leq n-2$, so $\Ric_k(F')>0$, and $F'$ has an effective $T^{l+1}$-action.
		Recall that $f\geq k-1$ and $k\geq 2$, and in the case that $k=2$, $M^n$ is assumed to be even-dimensional, so $f\geq 2$.
		Thus, in all cases, it follows that $\dim(F') = f+2 \geq 4$.
		So by the induction hypothesis, $l+1 \leq \lfloor\frac{f+2+k}{2}\rfloor - 1$, which implies that $l\leq \lfloor\frac{f+k}{2}\rfloor - 1$.
		Thus, $r \leq \frac{n-f}{2} + \lfloor\frac{f+k}{2}\rfloor - 1 \leq  \lfloor\frac{n+k}{2}\rfloor - 1$.
		Therefore, the result follows.
	\end{proof}	
	
	\section{Connectedness Principle}\label{sec:connprinc}
	
	In this section, we prove Theorem \ref{result:connprinc}. 
	Our argument is a modification of Guijarro and Wilhelm's approach in \cite{softconnprinc}. 
	We begin by setting up notation and terminology.
	Given an embedded submanifold $N$ of $M$, let $\Omega_N$ denote the space of piecewise-smooth curves in $M$, parametrized on the unit interval $[0,1]$, that start and end in $N$. 
	Now define the \textit{energy functional}
	\[
		E:\Omega_N\to [0,\infty), \qquad E(\gamma)\defeq \tfrac{1}{2}\int_0^1|\gamma'(t)|^2\;dt.
	\]
	It follows from the first variation of energy that critical points of the energy functional $E$ are geodesics that start and end perpendicular to $N$.
	To prove Theorem \ref{result:connprinc}, we will apply a Morse theoretic argument using a lower bound for the index of critical points of the energy functional $E$. 
	
	For the remainder of this section, let $\gamma:[0,b]\to M$ be a unit speed geodesic that starts and ends perpendicular to $N$. 
	We will use the ``\textit{index}'' of $\gamma$ to mean the index of $\gamma$ as a critical point of $E$.
	If we vary $\gamma$ by geodesics that start perpendicular to $N$, the corresponding collection of Jacobi fields forms a Lagrangian subspace of the collection of all Jacobi fields along $\gamma$ that are perpendicular to $\gamma'$.
	For this reason, we will now review Lagrangian subspaces of Jacobi fields, the Riccati operator on such subspaces, and the Transverse Jacobi Field Comparison from \cite{softconnprinc}.
	
	\subsection{Transverse Jacobi Field Comparison}
		
		
		Let $\gamma:[0,b]\to M$ be a unit speed geodesic that starts and ends perpendicular to an embedded submanifold $N$. 
		Considering geodesic variations of $\gamma$ that leave $N$ orthogonally at $t=0$, we will let $\mathcal{J}_N$ denote the vector space of corresponding Jacobi fields:
		\begin{align}
			\mathcal{J}_N &\defeq \{\text{Jacobi fields along }\gamma\text{ corresponding to variations}\nonumber \\
			&\hspace{21pt}\text{by geodesics that start perpendicular to }N\}.\label{eq:JN}
		\end{align}
		In particular, if we let $S_{\gamma'(0)}:T_{\gamma(0)}N \to T_{\gamma(0)}N$ denote the shape operator of $N$ determined by $\gamma'(0)$, then $J\in \mathcal{J}_N$ if and only if 
		\[
			J(0) \in T_{\gamma(0)}N, \qquad \text{and} \qquad (J'(0))^\top = S_{\gamma'(0)} J(0).
		\]
		Here, $(J'(0))^\top$ denotes the projection of $J'(0)$ to $T_{\gamma(0)}N$. For more information, see Chapter 10 Section 4 of \cite{docarmo}. 
		
		The set of times $t\in[0,b]$ for which $\{ J(t): J\in\mathcal{J}_N\}=\gamma'(t)^\perp$ is open and dense in $[0,b]$.
		Namely, these are the times at which no nontrivial Jacobi fields in $\mathcal{J}_N$ vanish. 
		For these values of $t$, there is a well-defined Riccati operator
		\[
			S_t:\gamma'(t)^\perp \to \gamma'(t)^\perp, \qquad S_t(v)=J_v'(t),
		\]
		where $J_v$ is the unique Jacobi field in $\mathcal{J}_N$ such that $J_v(t)=v$.
		The Jacobi equation can then be decomposed into two first-order equations:
		\[
			S_t J = J', \qquad S_t' + S_t^2 + R_{\gamma'} = 0.
		\]
		Here, $S_t'$ denotes the covariant derivative of $S$ along $\gamma$, and $R$ denotes the directional curvature operator along $\gamma$, namely $R_{\gamma'}(v) \defeq R(v,\gamma')\gamma'$.
		
		Now, given a subspace $\mathcal{V}\subseteq \mathcal{J}_N$ and a time $t\in[0,b]$, we can obtain a subspace $\mathcal{V}(t)\subseteq \gamma'(t)^\perp$ by setting
		\begin{equation}\label{eq:V(t)}
			\mathcal{V}(t) \defeq \{J(t):J\in\mathcal{V}\} \oplus \{J'(t):J\in\mathcal{V}\text{ and }J(t) = 0\}.
		\end{equation}
		The second summand, $\{J'(t):J\in\mathcal{V}\text{ and }J(t) = 0\}$, is trivial for almost all values of $t$, and the subspaces $\mathcal{V}(t)$ vary smoothly along $\gamma(t)$. 
		Next, we recall the following terminology introduced by Guijarro and Wilhelm in \cite{softconnprinc}:
		\begin{definition}
			Given $t\in[0,b]$, we say that a subspace $\mathcal{V}\subseteq \mathcal{J}_N$ is of \textit{full index} at $t$ if any field in $\mathcal{J}_N$ that vanishes at time $t$ is an element of $\mathcal{V}$. We say $\mathcal{V}$ is of \textit{full index} on an interval $I$ if it is of full index at $t$ for all $t\in I$.
		\end{definition}
		Spaces of Jacobi fields of full index are useful for finding subspaces on which the Riccati operator has negative trace.
		For simplicity of notation, given a time $t\in[0,b]$ and a subspace $W\subseteq \gamma'(t)^\perp$, we let $\trace S_t|_W$ denote the trace of the Riccati operator restricted to $W$ composed with the projection onto $W$.
		In the proof of Theorem \ref{result:connprinc}, we will use the following:
		
		\begin{lemma}[Lemma 1.5 in \cite{softconnprinc}]\label{lem:comparison}
			Let $\gamma:[0,b]\to M$ be a unit speed geodesic in a complete Riemannian manifold $M$ with $\Ric_k\geq k$. Let $\mathcal{J}_N$ be the space Jacobi fields along $\gamma(t)$ defined in Equation \eqref{eq:JN}, let $S$ be the associated Riccati operator, and let $W_0\perp \gamma'(0)$ be a $k$-dimensional subspace such that $\trace S_0|_{W_0}\leq k\cdot \cot(s_0)$ for some $s_0\in(0,\pi)$. Let $\mathcal{V}$ denote the subspace of $\mathcal{J}_N$ formed by those Jacobi fields that are orthogonal to $W_0$ at $t=0$. Let $\mathcal{V}(t)^\perp$ denote the subspace of $\gamma'(t)^\perp$ that is orthogonal to $\mathcal{V}(t)$. If $\mathcal{V}$ is of full index on $[0,b]$, then for all $t\in[0,b]$,
			\[
				\trace S_t|_{\mathcal{V}(t)^\perp}\leq k\cdot \cot(t+s_0).
			\]
		\end{lemma}
		We now define a few subspaces of $\mathcal{J}_N$ to be used later which have full index on different subsets of $[0,b]$:
		\begin{align}
			\mathcal{K} &\defeq \mathrm{span} \{J\in\mathcal{J}_N : J(t) = 0 \text{ for some } t\in [0,b] \}\label{eq:K}\\
			\mathcal{K}_+ &\defeq \mathrm{span} \{J\in\mathcal{J}_N : J(t) = 0 \text{ for some } t\in (0,b] \}\label{eq:K+}\\
			\mathcal{K}_b &\defeq \{J\in\mathcal{J}_N: J(b)=0\}\label{eq:Kb}
		\end{align}
		Notice $\mathcal{K}_+$ are the Jacobi fields that create focal points for $N$ on $(0,b]$. 
		Because a given Jacobi field in $\mathcal{K}_+$ can vanish multiple times in $(0,b]$, the number of focal points for $N$ on $(0,b]$ is bounded below by the dimension of $\mathcal{K}_+$. 
		Furthermore, notice the Riccati operator $S_t$ is well-defined on $\mathcal{K}(t)^\perp$ for every $t\in[0,b]$. 
		In particular, if there is a Jacobi field $\mathcal{J}_N$ that vanishes at $t\in[0,b]$, then for any vector $v\in\mathcal{K}(t)^\perp$, there could be many possible choices of Jacobi field $J_v\in\mathcal{J}_N$  such that $J_v(t)=v$. However, $J_v'(t)$ does not depend on the choice of $J_v$. 
		Finally, notice it follows from Equation \eqref{eq:V(t)} that $\mathcal{K}_b(b)=\{J'(b):J\in\mathcal{J}_N \text{ and }J(b)=0\}$.

	\subsection{Morse index theorem for endmanifolds}
		
		As mentioned earlier in this section, given a unit speed geodesic $\gamma:[0,b]\to M$ that starts and ends perpendicular to $N$, we will use variations of $\gamma$ by geodesics that start and end perpendicular to $N$ to obtain a lower bound on the index of $\gamma$.
		So define 
		\begin{align*}
			\mathcal{J}_{N,N}&\defeq\{\text{Jacobi fields along }\gamma\text{ corresponding to variations}\nonumber \\
			&\hspace{21pt}\text{by geodesics that start and end perpendicular to }N\}.
		\end{align*}
		
		Recall $S_{\gamma'(b)}:T_{\gamma(b)}N \to T_{\gamma(b)}N$ denotes the shape operator of $N$ determined by $\gamma'(b)$.
		Now, define $A:\mathcal{J}_{N,N}\times\mathcal{J}_{N,N}\to\mathbb{R}$ to be the symmetric bilinear form given by 
		\begin{equation}\label{eq:A}
			A(J_1,J_2)\defeq\left\langle J_1'(b)-S_{\gamma'(b)}J_1(b),J_2(b)\right\rangle.
		\end{equation}
		In particular, $A$ is the difference between the Riccati operator and the shape operator. 
		Hingston and Kalish proved the Morse index theorem for two endmanifolds in the case when each submanifold lies at a focal point of the other \cite{HingstonKalish}.
		Formulated for our current setting, their result can be written as follows:
		\begin{lemma}[\cite{HingstonKalish}]\label{lem:HK}
			Given a geodesic $\gamma:[0,b]\to M$ that starts and ends perpendicular to $N$, 
			\[
			\index\gamma=\index A+\text{number of focal points in }(0,b]-\dim(\mathcal{K}_b(b)\cap TN^\perp),
			\]
			where the number of focal points is counted with multiplicities. 
			In particular, because a given Jacobi field in $\mathcal{K}_+$ can vanish multiple times in $(0,b]$, 
			\[
				\index\gamma\geq \index A+\dim \mathcal{K}_+ -\dim(\mathcal{K}_b(b)\cap TN^\perp).
			\]
		\end{lemma}
		
		To obtain a lower bound for the index of $\gamma$, we will bound the index of $A$ by combining Lemma \ref{lem:comparison} with the following algebraic result:
		
		\begin{lemma}[Proposition 4.1 in \cite{softconnprinc}]\label{lem:negdef}
			Let $A:U\to U$ be a self-adjoint endomorphism on an $l$-dimensional inner product space. Suppose there exist $k\in\{1,2,\dots,l-1\}$ and $\lambda\in\mathbb{R}$ such that, for all $k$-dimensional subspaces $W\subset U$, $\trace A|_W\leq k\cdot \lambda$. Then there is an $(l-k+1)$-dimensional subspace $V\subseteq U$ such that for all unit $v\in V$, 
			\[
				\langle Av,v\rangle\leq \lambda.
			\]
		\end{lemma}
		
	\subsection{Connectedness principle for fixed point sets and $\mathbf{Ric_k>0}$}
	
		We are now ready to begin the proof of Theorem \ref{result:connprinc}, which we restate here for convenience:
		
		\begin{theorem}\label{thm:connprinc}
			Let $M^n$ be a compact, $n$-dimensional manifold with $\Ric_k>0$ for some $k\in\{2,\dots,n-1\}$. 
			Suppose $N^{n-d}$ is a compact embedded submanifold of codimension $d$ in $M^n$. If there is a Lie group $G$ that acts by isometries on $M^n$ and fixes $N^{n-d}$ point-wise, then the inclusion 
			\[
				N^{n-d}\hookrightarrow M^n\text{ is }(n-2d+2-k+\delta(G))\text{-connected},
			\] 
			where $\delta(G)$ is the dimension of the principal $G$-orbits in $M^n$.
		\end{theorem}
		
		First, we obtain the appropriate lower bound on the index of an energy-minimizing geodesic:
		
		\begin{lemma}\label{lem:indexofgeodesic}
			Let $M^n$, $N^{n-d}$, $G$, and $\delta(G)$ be defined as in Theorem \ref{thm:connprinc}. 
			Suppose $\gamma:[0,b]\to M$ is a unit speed geodesic that begins and ends perpendicular to $N^{n-d}$.
			Then
			\[
				\index \gamma \geq n - 2d + 2 - k + \delta(G).
			\] 
		\end{lemma}
		
		\begin{proof}
			If $\gamma$ passes through a principal orbit at some point, then the action fields for the $G$-action will contribute a value of $\delta (G)$ to the index count for $\gamma$ in the argument below.
			However, we must account for the possibility that $\gamma$ does not pass through any principal orbits. 
			Define the closed subgroup $H<G$ to be the intersection of the isotropy groups $G_{\gamma(t)}$ for all $t\in[0,b]$.
			Then $\gamma$ lies in a component $F^{n-l}$ of the fixed point set of the $H$-action on $M^n$.
			In the case that $H$ is trivial, $F^{n-l}$ is the entire manifold $M^n$, and $\gamma$ passes through principal orbits on an open dense subset of $(0,b)$. 
			In the general case, because $F^{n-l}$ is fixed point-wise by $H$, $F^{n-l}$ is invariant under the action of the normalizer $N(H)$ of $H$ in $G$. 
			Let $\delta(N(H))$ denote the dimension of the principal $N(H)$-orbits in $F^{n-l}$.
%
%
			Notice that $\delta(N(H)) + l \geq \delta(G)$.
			Thus to show that the index of $\gamma$ is bounded from below by $n-2d+2-k+\delta(G)$, it suffices to show that the index is at least  $n-2d+2-k+\delta(N(H))+l$.
			Throughout this proof, all orthogonal complements (e.g. $\mathcal{K}(t)^\perp$, $T_pN^\perp$, etc.) will be taken in the tangent bundle of $F^{n-l}$.
			
			Recall the definition of $\mathcal{K}_+$ from Equation \eqref{eq:K+}.
			If $\dim(\mathcal{K}_+)\geq n-d-k+\delta(N(H))+1$, then by Lemma \ref{lem:HK} and because $\mathcal{K}_b(b)\subseteq \gamma(b)^\perp$, we have
			\begin{align*}
				\index\gamma&=\index A+\text{number of focal points in }(0,b]-\dim(\mathcal{K}_b(b)\cap T_{\gamma(b)}N^\perp)\nonumber \\ 
				&\geq\dim\mathcal{K}_+-\dim(\gamma'(b)^\perp\cap T_{\gamma(b)}N^\perp)\nonumber\\ 
				&\geq (n-d-k+\delta(N(H))+1)-(d-l-1)\nonumber\\ 
				&=n-2d+2-k+\delta(N(H))+l.\label{eq:dimK+large}
			\end{align*}
			So if $\dim(\mathcal{K}_+)\geq n-d-k+\delta(N(H))+1$, then the result follows.
			
			Suppose now that $\dim(\mathcal{K}_+) \leq n - d - k + \delta(N(H))$.
			Because $\mathcal{J}_N$ is the collection of Jacobi fields along $\gamma$ which correspond to variations by geodesics that leave $N$ orthogonally at $t=0$, the Riccati operator on $N$, $S_0|_{T_{\gamma(0)}N}$, is precisely the shape operator on $N$, $S_{\gamma'(0)}$. Thus, because $N$ is totally geodesic, we have $S_0|_{T_{\gamma(0)}N}\equiv 0$. So if we define 
			\[
				U_0 \defeq \mathcal{K}_+(0)^\perp \cap T_{\gamma(0)}N,
			\]
			we have $S_0|_{U_0}\equiv 0$. Now let
			\[
				\Delta \defeq \{J\in\mathcal{J}_{N}:J\text{ is an action field for the }N(H)\text{-action on }F\}.
			\]
			Notice that if $J\in\Delta$, then $J(0)=0$, $J(b)=0$, $J'(0)\in T_{\gamma(0)}N^\perp$ and $J'(b)\in T_{\gamma(b)}N^\perp$.  
			In particular, $\Delta(0)\subseteq \mathcal{K}_+(0)\cap T_{\gamma(0)}N^\perp$.
			It follows that $\dim(\mathcal{K}_+(0)^\perp + T_{\gamma(0)}N)\leq \dim F-\dim \Delta(0)$. 
			Hence,
			\begin{align}
				\dim U_0&=\dim\mathcal{K}_+(0)^\perp+\dim T_{\gamma(0)}N-\dim(\mathcal{K}_+(0)^\perp+T_{\gamma(0)}N)\nonumber\\
				&\geq (\dim F-\dim\mathcal{K}_+(0))+\dim N-\left(\dim F-\dim\Delta(0)\right)\nonumber\\
				&=\dim N+\dim\Delta(0)-\dim\mathcal{K}_+(0)\label{eq:U0}\\
				&\geq(n-d)+\delta(N(H))-(n-d-k+\delta(N(H)))\nonumber\\
				&=k.\nonumber
			\end{align}
			Then for every $k$-dimensional subspace $W_0\subset U_0$, $S_0|_{W_0}\equiv 0 = k\cdot \cot(\pi/2)$.
			Let $\mathcal{V}\subset \mathcal{J}_N$ be the collection Jacobi fields that are orthogonal to $U_0$ at $\gamma(0)$.
			Then $\mathcal{V}$ is of full index on $[0,b]$, and $\dim(\mathcal{V}(t)^\perp)=\dim(U_0)$.
			By Lemma \ref{lem:comparison}, for every $k$-dimensional subspace $W_b$ of $\mathcal{V}(b)^\perp$, we have $\trace S_b|_{W_b}\leq k\cdot \cot(b+\pi/2)<0$. 
			Define 
			\[
				U_b\defeq \mathcal{V}(b)^\perp \cap T_{\gamma(b)}N.
			\]
			Then $\trace S_b$ is also negative on every $k$-dimensional subspace of $U_b$. Notice that because $\mathcal{V}(b)^\perp\subseteq \mathcal{K}_b(b)^\perp \subseteq \gamma'(b)^\perp$, we have 
			\[
				\mathcal{V}(b)^\perp + T_{\gamma(b)}N \subseteq (\mathcal{K}_b(b) \cap TN^\perp )^\perp \cap \gamma'(b)^\perp.
			\]
			In particular, we have 
			\begin{align*}
				\dim U_b &= \dim \mathcal{V}(b)^\perp + \dim T_{\gamma(b)}N - \dim(\mathcal{V}(b)^\perp + T_{\gamma(b)}N)\\
				&\geq \dim U_0 + \dim N - \dim\left((\mathcal{K}_b(b) \cap TN^\perp )^\perp \cap \gamma'(b)^\perp\right).\\
				&\geq \dim U_0 + (n - d) - \left(n - l - 1 - \dim(\mathcal{K}_b(b) \cap TN^\perp )\right)\\
				&= l - d + 1 + \dim U_0 + \dim(\mathcal{K}_b(b) \cap TN^\perp ).
			\end{align*}
			Recall from \eqref{eq:U0} that $\dim U_0 \geq \dim N+\dim\Delta-\dim\mathcal{K}_+$. Thus,
			\begin{align*}
				\dim U_b &\geq l - d + 1 + ( \dim N + \dim\Delta - \dim\mathcal{K}_+ ) + \dim(\mathcal{K}_b(b) \cap TN^\perp )\\
				&= l - d + 1 + ( n - d + \delta(N(H)) - \dim\mathcal{K}_+ ) + \dim(\mathcal{K}_b(b) \cap TN^\perp )\\
				&= n - 2d + 1 + \delta(N(H)) + l - \dim \mathcal{K}_+ + \dim(\mathcal{K}_b(b) \cap TN^\perp ).
			\end{align*}
			Hence, by Lemma \ref{lem:negdef}, there exists a subspace $\widetilde{U}_b \subseteq U_b$ with
			\[
				\dim \widetilde{U}_b \geq n - 2d + 2 - k + \delta(N(H)) + l - \dim \mathcal{K}_+ + \dim(\mathcal{K}_b(b) \cap TN^\perp )
			\]
			such that for all unit vectors $v \in \widetilde{U}_b$,
			\[
				\langle S_b(v) , v \rangle \leq 0.
			\]
			Let $\widetilde{\mathcal{U}} \subset \mathcal{J}_{N,N}$ be the subspace such that $\widetilde{\mathcal{U}}(b) = \widetilde{U}_b$. Then the bilinear form $A$ is negative-definite on $\widetilde{\mathcal{U}}$. So by Lemma \ref{lem:HK}, 
			\begin{align*}
				\index\gamma &\geq \index A+\dim \mathcal{K}_+ -\dim(\mathcal{K}_b(b)\cap TN^\perp)\\
				&\geq \dim \widetilde{\mathcal{U}} + \dim \mathcal{K}_+ -\dim(\mathcal{K}_b(b)\cap TN^\perp)\\
				&\geq n - 2d + 2 - k + \delta(N(H)) + l. 
			\end{align*}
			As mentioned before, because $\delta(N(H)) + l \geq \delta(G)$, the result follows.
		\end{proof}
		
		Now we use Lemma \ref{lem:indexofgeodesic} to prove Theorem \ref{thm:connprinc}:
		
		\begin{proof}[Proof of Theorem \ref{thm:connprinc}]
			Let $M^n$ be a compact, $n$-dimensional manifold with $\Ric_k>0$ for some $k\in\{2,\dots,n-1\}$, let $N^{n-d}$ be a compact embedded submanifold of codimension $d$ in $M^n$, and suppose a Lie group $G$ that acts by isometries on $M^n$ fixes $N^{n-d}$ point-wise.
			We must show that the inclusion $N^{n-d}\hookrightarrow M^n$ is $(n-2d+2-k+\delta(G))$-connected.
			
			As in the introduction to this section, let $\Omega_N$ denote the space of piecewise-smooth curves in $M$, parametrized on the unit interval $[0,1]$, that start and end in $N$. 
			Also recall the energy functional
			\[
				E:\Omega_N\to [0,\infty), \qquad E(\gamma)\defeq \tfrac{1}{2}\int_0^1|\gamma'(t)|^2\;dt.
			\]
			Notice $N$ embeds in $\Omega_N$ as the set of constant paths, and $E^{-1}(0) = N$.
			We claim that the inclusion $N\hookrightarrow \Omega_N$ is $(n-2d+1-k+\delta(G))$-connected.
			
			Recall that critical points of the energy functional $E$ are geodesics that start and end perpendicular to $N$.
			By Lemma \ref{lem:indexofgeodesic}, we have that the index of such a critical point is bounded below by $n-2d+2-k+\delta(G)$.
			Say that a critical point $\gamma$ has energy $E(\gamma) = e_0$. 
			$E^{-1}([0,e_0]) \subset \Omega_N$ can be approximated by a finite-dimensional submanifold of broken geodesics in $\Omega_N$. 
			Furthermore, on this approximation of $E^{-1}([0,e_0])$, one can find a Morse function that is $C^\infty$-close to $E$, is identical to $E$ on a neighborhood of $E^{-1}(0) = N$, and any critical point is non-degenerate and has index at least $n-2d+2-k+\delta(G)$.	It then follows that, up to homotopy, $\Omega_N$ can be obtained from $N$ by attaching cells of dimension at least $n-2d+2-k+\delta(G)$. 
			For more information, see \cite[Part III Section 16]{morse}.
			Thus, the inclusion $N\hookrightarrow \Omega_N$ is $(n-2d+1-k+\delta(G))$-connected.
			
			Finally, because $\pi_i(M,N)\cong\pi_{i-1}(\Omega_N,N)$, it follows from the long exact sequence of a pair in homotopy that $N\hookrightarrow M$ is $(n-2d+2-k+\delta(G))$-connected.
		\end{proof}	
	
	\section{$\mathbf{Ric_2>0}$ with large symmetry rank in odd dimensions}\label{sec:odd}
	
	In this section, we will study closed, simply connected, odd-dimensional manifolds with $\Ric_2>0$ and large symmetry rank.
	In particular, we obtain a diffeomorphism classification for those with maximal symmetry rank (the odd-dimensional case of Theorem \ref{result:main}), and we obtain a homeomorphism classification of those with approximately $3/4$-maximal symmetry rank (Theorem \ref{result:odd}).
	
	\subsection{Diffeomorphism classification}\label{sec:odddiffeo}
		
		In this section, we will prove the odd-dimensional case in Theorem \ref{result:main}, which we restate here for convenience:
		
		\begin{theorem}\label{thm:oddmaxsymrank}
			Let $M^n$ be a closed, simply connected, odd-dimensional Riemannian manifold with $\Ric_2>0$.
			If a torus $T^r$ of rank $r = \frac{n+1}{2} $ acts effectively and by isometries on $M^n$, then $M^n$ is diffeomorphic to $S^n$.
		\end{theorem}
		
		As mentioned in the introduction, Grove and Searle used Alexandrov geometry of positively curved orbit spaces to establish their classification of closed, connected $n$-manifolds with positive sectional curvature and symmetry rank $\lfloor\frac{n+1}{2}\rfloor$. 
		Because sectional curvatures are allowed to be negative when $\Ric_2>0$, we rely on Theorem \ref{result:connprinc}, algebraic topology, and a result by Montgomery and Yang \cite{MontgomeryYang} concerning circle actions on homotopy spheres to establish the diffeomorphism classification in Theorem \ref{thm:oddmaxsymrank}, rather than Alexandrov geometry.
		
		\begin{remark}\label{rmk:hamilton}
			Hamilton used Ricci flow in \cite{hamilton} to show that all closed Ricci-positive $3$-manifolds admit metrics of constant positive sectional curvature, and hence are spherical space forms.
			Thus, Theorem \ref{thm:oddmaxsymrank} holds in dimension $n=3$.
		\end{remark}
		
		In light of Remark \ref{rmk:hamilton}, we focus on odd dimensions $n\geq 5$.
		To prove Theorem \ref{thm:oddmaxsymrank}, first we establish that torus actions of maximal rank must have a circle subaction with fixed point set of codimension $2$:
		
		\begin{lemma}\label{lem:codim2}
			Suppose $M^n$ is a closed Riemannian manifold of odd dimension $n\geq 5$ with $\Ric_2>0$. 
			If a torus $T^r$ of rank $r =  \frac{n+1}{2} $ acts effectively and by isometries on $M^n$, then there is a subgroup $S^1 \subset T^r$ whose fixed point set has a connected component $N^{n-2}$ of codimension 2, and the action by $T^{r-1} \defeq T^r/S^1$ on $N^{n-2}$ is effective.
		\end{lemma}
		
		
		\begin{proof}
			Because $n\geq 5$, $r=\frac{n+1}{2}\geq 3$. Thus by Proposition \ref{result:torusfixedpt}, there exists at least one circle subgroup of $T^r$ with non-empty fixed point set. Among the collection of components of fixed point sets of all circle subgroups of $T^r$, choose an element $N$ that has minimal codimension in $M^n$, and let $S^1$ denote the circle subgroup that fixes $N$. Then $N$ has even codimension in $M^n$, and $T^{r-1}\defeq T^r/S^1$ acts on $N$. 
			
			We first prove that the codimension of $N$ must be $2$.
			Because $r\geq 3$, we have $r-1\geq 2$. So if $\dim(N)=1$, then the $T^{r-1}$-action on $N$ has kernel of rank at least $1$, and hence $N$ is fixed by a $T^2\subset T^r$. Thus, by Proposition \ref{prop:mincodim}, the codimension of $N$ is not minimal, which is a contradiction.
			Suppose now that $3\leq \dim(N)\leq n-4$. Then because $N$ is totally geodesic, $\Ric_2(N)>0$. Thus because $r-1=\frac{n-1}{2}>\frac{(n-4)+1}{2} \geq \frac{\dim(N) +1}{2}$, the $T^{r-1}$-action on $N$ has kernel of dimension at least $1$ by Proposition \ref{result:symrankbound}. So again, by Proposition \ref{prop:mincodim}, the codimension of $N$ is not minimal, which is a contradiction. 
			Therefore, the codimension of $N$ must be $2$.
			
			Now we will show that the $T^{r-1}$-action on $N^{n-2}$ is effective.
			By Proposition \ref{prop:mincodim}, the kernel of the $T^{r-1}$-action on $N^{n-2}$ is at most finite.
			If the kernel is non-trivial, because $N^{n-2}$ is also fixed by $S^1$, it follows that $N^{n-2}$ is fixed by a subgroup of $T^r$ of the form $\mathbb{Z}_m\times \mathbb{Z}_m$ for some natural number $m$.
			Let $\nu_xN^{n-2}$ denote the $2$-dimensional normal space to $N$ at an arbitrary point $x$.
			Then we have a faithful representation $\mathbb{Z}_m\times \mathbb{Z}_m \to \mathsf{GL}(\nu_xN^{n-2})$, which is only possible if $m = 2$.
			Thus, there exists an involution $\sigma \in T^r$ that fixes $N^{n-2}$, and it follows that the component $F$ of the fixed point set of $\sigma$ containing $N^{n-2}$ is of codimension $1$ in $M^n$.
			$F^{n-1}$ is invariant under the $T^r$-action, and because $F^{n-1}$ is totally geodesic, it follows that $\Ric_2(F^{n-1})>0$.
			Then by Proposition \ref{result:symrankbound}, there must exist a circle subgroup of $T^r$ that fixes $F^{n-1}$, which contradicts that fact that $N^{n-2}$ was chosen to have minimal codimension.
			Therefore, the $T^{r-1}$-action on $N^{n-2}$ must be effective.
		\end{proof}
		
		Now, we recall the following consequence of Poincar\'e Duality:
		
		\begin{lemma}[Lemma 2.2 in \cite{connprinc}]\label{lem:periodic}
			Let $M^n$ and $N^{n-d}$ be connected, closed, orientable manifolds.
			Suppose the inclusion $N^{n-d}\hookrightarrow M^n$ is $(n-d-l)$-connected, with $n-d-2l>0$.
			Let $[N]\in H_{n-d}(M^n;\mathbb{Z})$ denote the image of the fundamental class of $N$, and let $e\in H^d(M^n;\mathbb{Z})$ denote its Poincar\'e dual.
			Then the homomorphisms $\cup e: H^i(M;\mathbb{Z})\to H^{i+d}(M;\mathbb{Z})$ given by $x\mapsto x\cup e$ are surjective for $l\leq i<n-d-l$ and injective for $l<i\leq n-d-l$.
		\end{lemma}
		
		The following is a simple consequence of Lemma \ref{lem:periodic}; for details, see Fang and Rong's proof of their Lemma 4.2 in \cite{FangRong2004}.
		
		\begin{corollary}\label{cor:homotopysphere}
			Suppose $M^n$ is a closed, odd-dimensional, simply connected, smooth manifold.
			Assume $M^n$ contains a closed, connected, embedded submanifold $N^{n-2}$ of codimension $2$ such that the inclusion $N^{n-2}\hookrightarrow M^n$ is $(n-3)$-connected.
			Then $M^n$ and $N^{n-2}$ are both homotopy equivalent to spheres.
		\end{corollary}
		
		Finally, we recall a result proven by Montgomery and Yang.\footnote{Montgomery and Yang originally excluded dimensions $4$, $5$, and $6$, but we now know that their argument also holds in these dimensions thanks to the resolution of the Poincar\'e conjecture.}
		
		\begin{lemma}[Proposition 3 in \cite{MontgomeryYang}]\label{lem:MontgomeryYang}
			Suppose $M^n$ is a homotopy sphere, and assume the circle $S^1$ acts smoothly on $M^n$ such that the fixed point set $N^{n-2}$ is simply connected and of codimension $2$ in $M^n$.
			Then $M^n$ is diffeomorphic to the standard sphere $S^n$ such that the $S^1$-action on $M^n$ is smoothly equivalent to a linear circle action on $S^n$.
		\end{lemma}
		
		We are now ready to classify odd-dimensional manifolds with $\Ric_2>0$ and maximal symmetry rank.

		\begin{proof}[Proof of Theorem \ref{thm:oddmaxsymrank}]
			We will prove Theorem \ref{thm:oddmaxsymrank} by induction on the dimension $n$.
			For the base case, as mentioned in Remark \ref{rmk:hamilton}, the only simply connected, closed $3$-manifold with $\Ric_2>0$ is $S^3$.
			Now, assume Theorem \ref{thm:oddmaxsymrank} holds in odd dimensions $3,\dots,n-2$ for some $n\geq 5$, let $(M^n,g)$ be a closed, simply connected, odd-dimensional Riemannian manifold with $\Ric_2>0$, and suppose a torus $T^r$ of rank $r =  \frac{n+1}{2}$ acts effectively and by isometries on $M^n$.
			By Lemma \ref{lem:codim2}, there is a subgroup $S^1\subset T^r$ whose fixed point set has a connected component $N^{n-2}$ of codimension 2 such that $T^{r-1}\defeq T^r/S^1$ acts effectively on $N^{n-2}$.
			Because $N^{n-2}$ is totally geodesic and $n-2\geq 3$, we have that $\Ric_2(N^{n-2})>0$.
			Thus, by the induction hypothesis, $N^{n-2}$ is diffeomorphic to $S^{n-2}$.
			By Theorem \ref{result:connprinc}, the inclusion $N^{n-2}\hookrightarrow M^n$ is $(n-3)$-connected.
			Then by Corollary \ref{cor:homotopysphere}, $M^n$ is homotopy equivalent to a sphere.
			It is well-known that fixed point sets of circle actions on homotopy spheres are integral cohomology spheres; see for example \cite{Borel}.
			In particular, $N^{n-2}$ constitutes the entire fixed point set of the circle action.
			Therefore, by Lemma \ref{lem:MontgomeryYang}, $M^n$ is diffeomorphic to $S^n$.
		\end{proof}
	
	\subsection{Homeomorphism classification}\label{sec:oddhomeo}
		
		In this section, we will prove Theorem \ref{result:odd}, which we restate here for convenience:
		
		\begin{theorem}\label{thm:3n+10/8}
			Let $M^n$ be a closed, simply connected, Riemannian manifold of odd-dimension $n\geq 7$ with $\Ric_2>0$.
			Suppose a torus $T^r$ of rank $r\geq\tfrac{3n+10}{8}$ acts effectively and by isometries on $M^n$.
			Then $M^n$ is homeomorphic to $S^n$.
		\end{theorem}
		
		Our approach is an adaptation of the one outlined in \cite[Exercises 8.4.15--17]{petersen}.
		For odd dimensions $7 \leq n \leq 13$, we have that $\frac{n+1}{2}=\lceil\frac{3n+10}{8}\rceil$. Therefore, by Theorem \ref{thm:oddmaxsymrank}, we have proven Theorem \ref{thm:3n+10/8} for dimensions $n$ satisfying $7 \leq n \leq 13$. 
		To prove Theorem \ref{thm:3n+10/8} for dimensions $n\geq 15$, we now establish the following:
		
		\begin{lemma}\label{lem:NRic2}
			Let $M^n$ be a closed $n$-manifold with $\Ric_2>0$ on which a torus $T^r$ acts isometrically and effectively. Among the connected components of fixed point sets of circle sub-actions on $M^n$, choose $N$ that is maximal under inclusion. If $n\geq 7$ and $r\geq\frac{3n+10}{8}$, then:
			\begin{enumerate}
				\item \label{item:1} $\dim(N)\geq \frac{3n-2}{4}$, and
				
				\item \label{item:2} either $\codim(N)=2$ or $\symrank(N)\geq \frac{3\dim(N)+10}{8}$.
			\end{enumerate}
		\end{lemma}
		
		\begin{proof}
			Notice because $n\geq 7$, we have $r\geq\lceil\frac{3n+10}{8}\rceil\geq 4$. Let $S^1\subset T^r$ be the circle subgroup that fixes $N$, and define $T^{r-1}\defeq T^r/S^1$. Because $N$ has minimal codimension, we have that the kernel of the $T^{r-1}$-action on $N$ is at most finite by Proposition \ref{prop:mincodim}. Thus $\dim(N)\geq r-1\geq 3$, and because $N$ is totally geodesic, $\Ric_2(N)>0$. Applying Proposition \ref{result:symrankbound} to $N$, we have 
			\[
				\tfrac{\dim(N)+1}{2}\geq \symrank(N)\geq r-1\geq \tfrac{3n+10}{8}-1,
			\]
			and Part \eqref{item:1} follows.
			
			To prove Part \eqref{item:2}, assume $\codim(N)\neq 2$ and $\symrank(N)< \frac{3\dim(N)+10}{8}$. Because the $T^r$ action on $M^n$ is effective and the codimension of $N$ must be even, we have that $\dim(N)\leq n-4$. Thus
			\[
				\symrank(N)<\tfrac{3\dim(N)+10}{8}\leq\tfrac{3(n-4)+10}{8}<\tfrac{3n+10}{8}-1\leq r-1.
			\]
			It follows that the kernel of the $T^{r-1}$-action on $N$ is at least 1-dimensional, which contradicts $N$ being chosen to have minimal codimension by Proposition~\ref{prop:mincodim}.
		\end{proof}
		
		We will now use Lemma \ref{lem:NRic2} and Theorem \ref{result:connprinc} to prove Theorem \ref{thm:3n+10/8} using the work of Smale in \cite{smale}.
		
		\begin{proof}[Proof of Theorem \ref{thm:3n+10/8}]
			We will prove Theorem \ref{thm:3n+10/8} by induction on $n$. 
			For the base case, notice that in off dimensions $n$ satisfying $7\leq n\leq13$, we have $\frac{n+1}{2}=\lceil\frac{3n+10}{8}\rceil$. 
			Thus, the base cases are established by Theorem \ref{thm:oddmaxsymrank}.			
			Suppose for the sake of induction that Theorem \ref{thm:3n+10/8} holds for odd dimensions up to $n-2$ for some $n\geq15$. We will show it also holds in dimension $n$.
			
			Let $M^n$ be a closed, simply connected, Riemannian manifold of odd-dimension $n\geq 15$ with $\Ric_2>0$, and suppose a torus $T^r$ of rank $r\geq\tfrac{3n+10}{8}$ acts effectively and by isometries on $M^n$.
			We will show that $M^n$ is homeomorphic to $S^n$.
			Because $r\geq \frac{3n+10}{8}>3$ when $n\geq 15$, it follows from Theorem \ref{result:torusfixedpt} that there are circle subgroups of $T^r$ whose fixed point sets on $M^n$ are non-empty. Among the connected components of fixed point sets of these circle sub-actions on $M^n$, choose $N$ that is maximal under inclusion. 
			
			By Lemma \ref{lem:NRic2}, $\dim(N)\geq\frac{3n-2}{4}$ and either $\codim(N)=2$ or $\symrank(N)\geq \frac{3\dim(N)+10}{8}$. If $\codim(N)=2$, then by Theorem \ref{result:connprinc} and Corollary \ref{cor:homotopysphere}, $M^n$ is homotopy equivalent to $S^n$, and the result follows from Smale's resolution to the Generalized Poincar\'e conjecture for dimensions $\geq 5$ in \cite{smale}. Suppose instead that $\codim(N)\geq 4$ and $\symrank(N)\geq \frac{3\dim(N)+10}{8}$. Then because $\dim(N)\geq \frac{3n-2}{4}$, we have
			\[
				n-2\codim(N)+1 \geq \tfrac{n}{2}
			\]
			Hence by Theorem \ref{result:connprinc}, the inclusion $N\hookrightarrow M$ is at least $\lceil\frac{n}{2}\rceil$-connected. Thus, because $M$ is simply connected, so is $N$. Because $n\geq 15$, we have $\dim(N)\geq\lceil\frac{3n-2}{4}\rceil\geq11$. Hence, the induction hypothesis implies that $N$ is homeomorphic to a sphere. Thus for $1\leq i\leq \frac{n-1}{2}$,  
			\[
				H_i(M^n)\cong H_i(N)\cong0.
			\]
			Applying Poincar\'e Duality, it follows that $H_i(M^n)\cong 0$ for $1\leq i\leq n-1$. Because $M^n$ is simply connected, it follows that $M^n$ is homotopy equivalent to a sphere, and again the result follows by the work of Smale in \cite{smale}.
		\end{proof}
			
	\section{$\mathbf{Ric_2>0}$ with large symmetry rank in even dimensions}\label{sec:even}
	
	In this section, we will study closed, simply connected, even-dimensional manifolds with $\Ric_2>0$ and large symmetry rank.
	In particular, we prove those with half-maximal symmetry rank have positive Euler characteristic (Theorem \ref{result:evenEuler}), we obtain a strong classification for those with maximal symmetry rank and bounded second Betti number (the even-dimensional case of Theorem \ref{result:main}), and we obtain a weaker classification of those with $3/4$-maximal symmetry rank and bounded second Betti number (Theorem \ref{result:even}).
	
	
	\subsection{Positive Euler characteristic}\label{sec:posEuler}
	
		First, we establish Theorem \ref{result:evenEuler}, which we restate here for convenience:
			
		\begin{theorem}\label{thm:Ric2evenEuler}
			Suppose $M^n$ is a closed Riemannian manifold of even dimension $n\geq 8$ with $\Ric_2>0$.  
			If a torus $T^r$ of rank $r\geq \frac{n}{4}+2$ acts effectively and by isometries on $M^n$, then $\chi(M^n)>0$.
		\end{theorem}
		
		To prove Theorem \ref{thm:Ric2evenEuler}, we will use the following topological observation:
		
		\begin{proposition}\label{prop:chainoftori}
			Suppose a torus $T^r$ acts isometrically and effectively on a closed manifold $M$. If the fixed point set $M^{T^r}$ is non-empty, then given any point $x\in M^{T^r}$, there exists a chain of subgroups $T^1\subset T^2\subset \dots\subset T^{r-1}\subset T^r$ such that the following inclusions of components of fixed point sets containing $x$ are each of minimal, positive, even codimension:
			\[
			M^{T^r}_x\subset M^{T^{r-1}}_x\subset \dots\subset M^{T^2}_x\subset M^{T^1}_x\subset M.
			\]
		\end{proposition}
		
		\begin{proof}
			Fix a point $x\in M^{T^r}$ and choose a circle subgroup $S_1^1\subset T^r$ such that the fixed point set component $M_{x}^{S_1^1}$ has minimal codimension in $M$. Because $T^r$ acts effectively on $M$, we have $M_{x}^{S_1^1}\neq M$. Set $T^1\defeq S_1^1$. Now choose a circle subgroup $S_2^1\subset T^r/T^{1}$ such that the fixed point set component $M_{x}^{T^2}$ for $T^2\defeq S_2^1\times T^1$ has minimal codimension in $M_{x}^{T^{1}}$. Because $M_{x}^{T^{1}}$ was chosen to have minimal codimension in $M$, $T^r/T^{1}$ must act almost effectively on $M_{x}^{T^{1}}$ by Proposition \ref{prop:mincodim}. In particular, $S_2^1$ does not fix  all of $M_{x}^{T^1}$, and hence $M_{x}^{T^2}\neq M_{x}^{T^1}$. 
			
			Now for $i\in\{2,\dots,r-1\}$, we inductively choose $S_{i+1}^1\subset T^r/T^{i}$ such that the fixed point set component $M_{x}^{T^{i+1}}$ for $T^{i+1}\defeq S_{i+1}^1\times T^i$ has minimal codimension in $M_{x}^{T^i}$. Because $M_{x}^{T^{i}}$ was chosen to have minimal codimension in $M_{x}^{T^{i-1}}$, again $T^r/T^{i}$ must act almost effectively on $M_{x}^{T^{i}}$ by Proposition \ref{prop:mincodim}. This shows that $S_{i+1}^1\subset T^r/T^i$ does not fix all of $M_{x}^{T^i}$, and therefore $M_{x}^{T^{i+1}}\neq M_{x}^{T^{i}}$ for $i\in\{2,\dots,r-1\}$. 
		\end{proof}
		
		To prove Theorem \ref{thm:Ric2evenEuler}, we will apply Proposition \ref{result:torusfixedpt} for a $T^r$-action on an $n$-manifold of $\Ric_2>0$ with $r\geq \frac{n}{4}+2$ to obtain a fixed point for some $T^{r-2}$-subaction. Applying Proposition \ref{prop:chainoftori} will then give us the following topological restriction:
		
		\begin{lemma}\label{lem:chainfixedptsets}
			Suppose $M^n$ is closed and even-dimensional, and a torus $T^r$ acts isometrically and effectively on $M^n$ with $r\geq\frac{n}{4}+2$. Assume a subgroup $T^{r-2}\subset T^r$ has non-empty fixed point set $M^{T^{r-2}}$ in $M$, and for any point $x\in M^{T^{r-2}}$, consider a chain of fixed point set components $M^{T^{r-2}}_x\subset M^{T^{r-3}}_x\subset \dots\subset M^{T^2}_x\subset M^{T^1}_x\subset M$ guaranteed by Proposition \ref{prop:chainoftori}. Then either 
			\begin{enumerate}
				\item $\dim(M_x^{T^{r-2}})=0$, or 
				
				\item at least one of the inclusions in the chain $M^{T^{r-2}}_x\subset M^{T^{r-3}}_x\subset \dots\subset M^{T^2}_x\subset M^{T^1}_x\subset M$ is of codimension 2.
			\end{enumerate}
		\end{lemma}
		
		\begin{proof}
			Suppose $\dim(M_x^{T^{r-2}})\neq0$. Let $d_i$ denote the codimension of $M_x^{T^i}$ in $M_x^{T^{i-1}}$ for $i\in\{2,\dots,r-2\}$, and let $d_1$ denote the codimension of $M_x^{T^1}$ in $M$. Because fixed point sets of torus actions have even codimension, it follows that $\dim(M_x^{T^{r-2}})\geq 2$ and each $d_i$ is even. Because the inclusions in the chain are proper, it follows that $d_i\geq 2$ for all $i$. Now if $d_i\geq 4$ for all $i$, then because $\dim(M_x^{T^{r-2}})\geq 2$, we have
			\begin{align*}
				n=\dim(M^n)&\geq \dim(M_x^{T^{r-2}})+d_{r-2}+d_{r-3}+\dots+d_2+d_1\\
				&\geq 2+4(r-2).
			\end{align*}
			It then follows that $r < \frac{n}{4}+2$, which contradicts the assumption that $r\geq\frac{n}{4}+2$. Therefore, $d_i=2$ for some $i\in\{1,\dots,r-2\}$.
		\end{proof}
		
		Finally, we recall the following topological result established by Conner:
		
		\begin{lemma}[\cite{Conner}]\label{lem:conner}
			If $T$ is a torus acting on a manifold $M$ with fixed point set $M^T$, then  $\chi(M)=\chi(M^T)$.
		\end{lemma}
		
		We are now ready to prove Theorem \ref{thm:Ric2evenEuler}.
		
		\begin{proof}[Proof of Theorem \ref{thm:Ric2evenEuler}]
			Suppose $M^n$ is closed, even-dimensional, has $\Ric_2>0$, and $T^r$ acts isometrically and effectively on $M^n$ with $r\geq\frac{n}{4}+2$. 
			Because $\Ric_2(M^n)>0$, Proposition \ref{result:symrankbound} states that $r\leq \frac{n}{2}$. 
			For even dimensions $n$, the inequalities $\frac{n}{4}+2\leq r\leq\frac{n}{2}$ are only consistent if $n\geq 8$. 
			In this case, $r\geq 4$. 
			Thus by Proposition \ref{result:torusfixedpt}, there exists a subgroup $T^{r-2}\subset T^r$ such that the fixed point set $M^{T^{r-2}}$ in $M$ is non-empty. By Lemma \ref{lem:conner}, it suffices to show that $\chi(M^{T^{r-2}})>0$. In particular, we will show that every connected component $M_x^{T^{r-2}}$ has positive Euler characteristic. 
			
			Choose an arbitrary point $x\in M^{T^{r-2}}$. By Proposition \ref{prop:chainoftori}, there exists a chain of subgroups $T^1\subset T^2\subset \dots\subset T^{r-1}\subset T^r$ such that the inclusions of fixed point set components containing $x$, $M^{T^r}_x\subset M^{T^{r-1}}_x\subset \dots\subset M^{T^2}_x\subset M^{T^1}_x\subset M$, are proper. 
			By Lemma \ref{lem:chainfixedptsets}, either $\dim(M_x^{T^{r-2}})=0$, or at least one of the inclusions in the chain $M^{T^{r-2}}_x\subset M^{T^{r-3}}_x\subset \dots\subset M^{T^2}_x\subset M^{T^1}_x\subset M$ is of codimension 2. 
			If $\dim(M_x^{T^{r-2}})=0$, then $\chi(M_x^{T^{r-2}})>0$, and we are done. 
			
			Now assume $\dim(M_x^{T^{r-2}})\geq 2$ and one of the inclusions  $M^{T^i}_x\subset M^{T^{i-1}}_x$ for $i\in\{1,\dots,r-2\}$ is of codimension 2.
			Here, we are using the convection that $T^0$ is the trivial subgroup and $M_x^{T^0}=M$. 
			We will show that $\chi(M^{T^{i-1}})>0$. 
			Let $m=\dim(M^{T^{i-1}}_x)$. Because $\dim(M_x^{T^{r-2}})\geq 2$, we have $m=2+\dim(M^{T^{i}})\geq 4$. So because $M^{T^{i-1}}_x$ is totally geodesic in $M$, we have $\Ric_2(M^{T^{i-1}}_x)>0$. Hence the Betti numbers for $b_1(M^{T^{i-1}}_x)$ and $b_{m-1}(M^{T^{i-1}}_x)$ are both zero. 
			So if $M^{T^{i-1}}$ is $4$-dimensional, $\chi(M^{T^{i-1}})>0$. 
			
			Suppose now that $M^{T^{i-1}}$ has dimension $m\geq 6$.
			By Theorem \ref{result:connprinc}, because $M^{T^i}_x$ is fixed by the $S^1\cong T^i/T^{i-1}$-action on $M^{T^{i-1}}_x$, the inclusion $M^{T^i}_x\hookrightarrow M^{T^{i-1}}_x$ is $(m-3)$-connected. 
			Thus by Lemma \ref{lem:periodic}, we have homomorphisms $H^i(M^{T^{i-1}}_x;\mathbb{Z})\to H^{i+2}(M^{T^{i-1}}_x;\mathbb{Z})$ which are surjective for $1\leq i< m-3$ and injective for $1<i\leq m-3$.
			The hypothesis that $n-d-2l>0$ in Lemma \ref{lem:periodic} is satisfied because $m\geq 6$, $d=2$, and $l=1$ in this case.
			Therefore, it follows that all of the odd Betti numbers of $M^{T^{i-1}}_x$ are zero, which implies $\chi(M^{T^{i-1}}_x)>0$. 
			
			Now for all dimensions $m\geq 4$, because $M^{T^{i-1}}_x$ is invariant under the $T^{r-2}$-action and $(M^{T^{i-1}}_x)^{T^{r-2}}=M^{T^{r-2}}_x$, it follows from Lemma \ref{lem:conner} that $\chi(M^{T^{r-2}}_x)>0$.
			Hence, we have shown that for all $x\in M^{T^{r-2}}$, the component $M_x^{T^{r-2}}$ containing $x$ has $\chi(M_x^{T^{r-2}})>0$. Therefore, $\chi(M^{T^{r-2}})>0$ and by Lemma \ref{lem:conner}, $\chi(M)>0$.
		\end{proof}
		
	\subsection{Classification for maximal symmetry rank}\label{sec:evenmaxsymrank}
		
		
		We will now prove the even dimensional case of Theorem \ref{result:main}.
		Specifically, we must prove if $M^n$ is closed, simply connected, even-dimensional, $n\neq 6$, $b_2(M^n)\leq 1$, $\Ric_2(M^n,g)>0$, and $\symrank(M^n,g) = \frac{n}{2}$, then $M^n$ is either diffeomorphic to $S^n$ or homeomorphic to $\mathbb{C}\mathrm{P}^{n/2}$.
		
		First, we note that the dimension $n=4$ case of Theorem \ref{result:main} follows from purely topological considerations, not relying on the curvature assumption.
		Orlik and Raymond prove in \cite{OrlikRaymond} that any closed, simply connected $4$-manifold $M^4$ with an effective $T^2$-action is equivariantly diffeomorphic to a connected sum of finitely many copies of $S^4$, $\pm\mathbb{C}\mathrm{P}^2$, or $S^2\times S^2$.
		Now if $b_2(M^4)\leq1$, then $2\leq \chi(M^4) \leq 3$.
		Thus, we have the following:
		
		\begin{corollary}\label{cor:dim4}
			Suppose $M^4$ is a closed, simply connected, $4$-dimensional manifold with a smooth, effective $T^2$-action.
			If $b_2(M^4)\leq 1$, then $M^4$ is equivariantly diffeomorphic to either $S^4$ or $\mathbb{C}\mathrm{P}^{2}$.
		\end{corollary}
		
		Sha and Yang proved in \cite{ShaYang} that any connected sum of finitely many copies of $S^4$, $\pm\mathbb{C}\mathrm{P}^2$, or $S^2\times S^2$ admits a metric of positive Ricci curvature. 
		This leads naturally to the following:
		
		\begin{question}
			Are there any closed, simply connected $4$-manifolds with $b_2\geq 2$ which admit metrics of $\Ric_2>0$ that are invariant under an effective $T^2$-action?
		\end{question}
		
		\begin{example}
			Recall from Example \ref{ex:S3xS3} that $S^2\times S^2$ admits a metric with $\Ric_2>0$ and symmetry rank $1$. 
			Hsiang and Kleiner prove in \cite{HsiangKleiner} that any closed, orientable, $4$-dimensional manifold with positive sectional curvature that has a nontrivial Killing field must be homeomorphic to $S^4$ or $\mathbb{C}\mathrm{P}^2$.
			Consequently, it is impossible for $S^2\times S^2$ to admit a metric of positive sectional curvature with symmetry rank $1$.
			It remains to be seen whether $S^2\times S^2$ can admit a metric with $\Ric_2>0$ and symmetry rank $2$.
		\end{example}
		
		Now we will establish Theorem \ref{result:main} for dimensions $n\geq 8$.
		Namely, we intend to prove the following:
		
		\begin{theorem}\label{thm:evenmaxsymrank}
			Let $M^n$ be a closed, simply connected Riemannian manifold of even dimension $n\geq 8$ with $\Ric_2>0$. 
			Suppose a torus $T^r$ of rank $r=\frac{n}{2}$ acts effectively and by isometries on $M^n$.
			\begin{enumerate}
				\item If $b_2(M^n)=0$, then $M^n$ is diffeomorphic to $S^n$.
				
				\item If $b_2(M^n)=1$, then $M^n$ is homeomorphic to $\mathbb{C}\mathrm{P}^{n/2}$. \label{item:CPn}
			\end{enumerate}
		\end{theorem}
		
		First, we will use Theorem \ref{thm:Ric2evenEuler} and Theorem \ref{result:connprinc} to justify the following:
		
		\begin{lemma}\label{lem:evenmaxsymrank}
			Let $M^n$ be a closed, simply connected Riemannian manifold of even dimension $n\geq 8$ with $\Ric_2>0$. 
			Suppose a torus $T^r$ of rank $r=\frac{n}{2}$ acts effectively and by isometries on $M^n$.
			Then there exist closed, simply connected, totally geodesic submanifolds $M^4 \subset M^6 \subset \dots \subset M^{n-2} \subset M^n$ and torus subgroups $T^1\subset T^2 \subset \dots \subset T^{r-1} \subset T^r$ such that
			\begin{enumerate}
				\item each submanifold $M^{2i}$ is fixed by the $T^{r-i}$ action,
				
				\item the action of $T^i \defeq T^r/T^{r-i}$ on $M^{2i}$ is effective for all $i$, and 
				
				\item each inclusion $M^{2i}\hookrightarrow M^{2i+2}$ is $(2i-1)$-connected.
			\end{enumerate}
		\end{lemma}
		
		\begin{proof}
			By Theorem \ref{thm:Ric2evenEuler}, because $r=\frac{n}{2}\geq \frac{n}{4}+2$ when $n\geq 8$, we have $\chi(M^n)>0$.
			Thus, the $T^r$-action on $M^n$ has non-empty fixed point set. 
			By Proposition \ref{prop:chainoftori}, given a point $x$ in $M^{T^r}$, there exists a chain of subgroups $T^1\subset T^2\subset \dots\subset T^{r-1}\subset T^r$ such that the following inclusions of components of fixed point sets containing $x$ are of positive, even codimension:
			\[
			M^{T^r}_x\subset M^{T^{r-1}}_x\subset \dots\subset M^{T^2}_x\subset M^{T^1}_x\subset M.
			\]
			Because $r=\frac{n}{2}$ and because each of the inclusions above are of positive, even codimension, it follows that $M^{T^r}_x$ is zero-dimensional and each inclusion $M_x^{T^j}\subset M_x^{T^{j-1}}$ is of codimension $2$.
			Define $M^{2i} = M_x^{T^{r-i}}$.
			Because each submanifold $M^{2i}$ is totally geodesic, we have $\Ric_2(M^{2i})>0$ for $2i\geq 4$.
			By Theorem \ref{result:connprinc}, each inclusion $M^{2i}\hookrightarrow M^{2i+2}$ is $(2i-1)$-connected, and because $M$ is simply connected, so are $M^{2i}$ for $2i\geq 4$.
			
			Now, we will show that the $T^i\defeq T^r/T^{r-i}$ action on $M^{2i}$ is effective for all $i$. 
			Suppose for the sake of contradiction that the kernel of the action is non-trivial for some $i$.
			By choosing the largest such index $i$, we may assume that the $T^{i+1}$ action on $M^{2i+2}$ is effective.
			As in the proof of Lemma \ref{lem:codim2}, it follows that there exists an involution $\sigma\in T^r$ that fixes $M^{2i}$, and the component $F$ of the fixed point set of $\sigma$ containing $x$ has dimension strictly larger than $M^{2i}$.
			Let $S^1\subset T^r$ be any circle subgroup containing $\sigma$. 
			Then the codimension of the fixed point set component $M^{S^1}_x$ in $F$ is even, and the codimension of $M^{S^1}_x$ in $M^{2i+2}$ is also even.
			Thus, it follows that the codimension of $F$ in $M^{2i+2}$ must be even and smaller than 2.
			This implies that the codimension of $F$ in $M^{2i+2}$ is zero, and hence the $T^{i+1}$ action on $M^{2i+2}$ is not effective, which is a contradiction.
			Therefore, the $T^i$ action on $M^{2i}$ must be effective for all $i$.
		\end{proof}

		We are now prepared to classify even-dimensional manifolds with $b_2\leq 1$, $\Ric_2>0$, and maximal symmetry rank.
		
		\begin{proof}[Proof of Theorem \ref{thm:evenmaxsymrank}]
			Let $M^n$ be a closed, simply connected Riemannian manifold of even dimension $n\geq 8$ with $\Ric_2>0$, and suppose a torus $T^r$ of rank $r=\frac{n}{2}$ acts effectively and by isometries on $M^n$.
			By Lemma \ref{lem:evenmaxsymrank}, there exist closed, simply connected, totally geodesic submanifolds $M^4 \subset M^6 \subset \dots \subset M^{n-2} \subset M^n$ and torus subgroups $T^1\subset T^2 \subset \dots \subset T^{r-1} \subset T^r$ such that each submanifold $M^{2i}$ is fixed by the $T^{r-i}$ action, the action of $T^i \defeq T^r/T^{r-i}$ on $M^{2i}$ is effective for all $i$, and each inclusion $M^{2i}\hookrightarrow M^{2i+2}$ is $(2i-1)$-connected.
			Thus $b_2(M^4) = b_2(M^6) = \dots = b_2(M^{n-2}) = b_2(M^n)$.
			So if $b_2(M^n) \leq 1$, then $b_2(M^4) \leq 1$, and it follows from Corollary \ref{cor:dim4} that $M^4$ is equivariantly diffeomorphic to either $S^4$ or $\mathbb{C}\mathrm{P}^2$.
			
			Because $M^4\hookrightarrow M^6$ is $3$-connected, we have that $\pi_3(M^6)=0$, and $\pi_2(M^6) = 0$ (resp. \!$\mathbb{Z}$) if $M^4\cong S^4$ (resp. \!$\mathbb{C}\mathrm{P}^2$).
			It follows from Poincar\'e duality and the Hurewicz theorem that $M^6$ is homotopy equivalent to $S^6$ or $\mathbb{C}\mathrm{P}^3$.
			By iterating the same argument for $M^8,\dots,M^n$, we conclude that $M^n$ is homotopy equivalent to $S^n$ if $b_2(M^n)=0$, or $M^n$ is homotopy equivalent to $\mathbb{C}\mathrm{P}^{n/2}$ if $b_2(M^n)=1$.
			
			If $b_2(M^n)=0$, then similar to the proof of Theorem \ref{thm:oddmaxsymrank}, by inductively applying Lemma \ref{lem:MontgomeryYang} to the submanifolds $M^4\subset M^6 \subset \dots \subset M^n$, it follows that $M^n$ is diffeomorphic to $S^n$.
			
			Suppose instead that $b_2(M^n)=1$, and hence $M^4$ is homeomorphic to $\mathbb{C}\mathrm{P}^2$ and $M^6,\dots,M^n$ are homotopy complex projective spaces.
			Fang and Rong proved that any homotopy $\mathbb{C}\mathrm{P}^n$ with a submanifold homeomorphic to $\mathbb{C}\mathrm{P}^{n-1}$ such that the inclusion map is at least $3$-connected must be homeomorphic to $\mathbb{C}\mathrm{P}^n$ \cite{FangRong2004}.
			Because $M^4$ is homeomorphic to $\mathbb{C}\mathrm{P}^2$ and each inclusion $M^{2i}\hookrightarrow M^{2i+1}$ is $(2i-1)$-connected, we can apply Fang and Rong's result to each inclusion $M^4 \subset M^6 \subset \cdots \subset M^n$, concluding that $M^n$ is homeomorphic to $\mathbb{C}\mathrm{P}^{n/2}$.
		\end{proof}

		\subsection{Classification for $\mathbf{3/4}$-maximal symmetry rank}\label{sec:even3/4}
			
			In this final section, we will prove Theorem \ref{result:even}, which we restate here for convenience:
			
			\begin{theorem}\label{thm:even3/4maxsymrank}
				Let $M^n$ be a closed, simply connected, Riemannian manifold of even dimension $n\geq 8$ with $\Ric_2>0$.
				Suppose a torus $T^r$ of rank $r\geq\tfrac{3n+6}{8}$ acts effectively and by isometries on $M^n$.
				Then $H^i(M^n;\mathbb{Z})=0$ for all odd values of $i$.
				Furthermore:
				\begin{enumerate}
					\item If $b_2(M^n)=0$, then $M^n$ is homeomorphic to $S^n$.
					
					\item If $b_2(M^n)=1$, then $M^n$ is tangentially homotopy equivalent to $\mathbb{C}\mathrm{P}^{n/2}$.
				\end{enumerate}
			\end{theorem}
			
			First, we establish an even dimensional analogue of Lemma \ref{lem:NRic2}:
			
			\begin{lemma}\label{lem:NRic2even}
				Let $M^n$ be a closed even-dimensional manifold with $\Ric_2>0$ on which a torus $T^r$ acts isometrically and effectively with non-empty fixed point set. 
				If $r\geq\frac{3n+6}{8}$, then there exists a connected submanifold $N\subset M$ of minimal codimension fixed by a circle subgroup of $T^r$ such that
				\begin{enumerate}
					\item \label{item:1even} $\dim(N)\geq \frac{3n-2}{4}$, and
					
					\item \label{item:2even} either $\codim(N)=2$ or $\symrank(N)\geq \frac{3\dim(N)+6}{8}$.
				\end{enumerate}
			\end{lemma}
			
			\begin{proof}
				Suppose the $T^r$-action on $M^n$ has non-empty fixed point set, and consider the chain of fixed point set components $M^{T^r}_x\subset M^{T^{r-1}}_x\subset \dots\subset M^{T^2}_x\subset M^{T^1}_x\subset M$ guaranteed by Proposition \ref{prop:chainoftori}.
				We will choose $N$ to be $M_x^{T^1}$.
				Because the inclusions $M_x^{T^i}\subset M_x^{T^{i-1}}$ are each of positive, even codimension, we have that $\dim(N)\geq 2r-2$, and because $r\geq \frac{3n+6}{8}$, it follows that $\dim(N) \geq \frac{3n-2}{4}$. This proves Part \eqref{item:1even}.
				
				To prove Part \eqref{item:2even}, assume $\codim(N)\neq 2$ and $\symrank(N)< \frac{3\dim(N)+6}{8}$. 
				Recall from Proposition \ref{prop:chainoftori} that the inclusion $N\subset M^n$ is minimal in the sense that if a circle subgroup of $T^r$ fixes another connected submanifold $N'\subset M^n$, then $\codim(N')\geq \codim(N)$.
				Now, because the $T^r$-action on $M^n$ is effective and the codimension of $N\subset M^n$ must be even, we have that $\dim(N)\leq n-4$. Thus
				\[
					\symrank(N)<\tfrac{3\dim(N)+6}{8}\leq\tfrac{3(n-4)+6}{8}<\tfrac{3n+6}{8}-1\leq r-1.
				\]
				So setting $T^{r-1}=T^r/T^1$, it follows that there is a circle subgroup of $T^{r-1}$ that fixes $N$, meaning that $N$ is fixed by a two-dimensional torus subgroup of $T^r$. 
				By Proposition \ref{prop:mincodim}, this implies that there is a circle subgroup of $T^r$ that fixed a submanifold of larger dimension than $N$, which is a contradiction.
				Therefore, Part \eqref{item:2even} follows.
			\end{proof}
			
			In proving Theorem \ref{thm:even3/4maxsymrank}, if the submanifold $N$ from Lemma \ref{lem:NRic2even} is of codimension $2$, then we will apply the following:
			
			\begin{lemma}\label{lem:codim2even}
				Suppose $M^n$ is a closed, even-dimensional, simply connected, smooth manifold of dimension $n\geq 4$. 
				Assume $M^n$ contains a closed, connected, embedded submanifold $N^{n-2}$ of codimension $2$ such that the inclusion $N^{n-2}\hookrightarrow M^n$ is $(n-3)$-connected.
				Then $H^i(M^n;\mathbb{Z})=0$ for all odd values of $i$.
			\end{lemma}
			
			\begin{proof}
				In dimension $n=4$, the result follows from Poincar\'e duality.
				Suppose $n\geq 6$.
				Because $M^n$ is simply connected, so is $N^{n-2}$.
				Let $[N]\in H_{n-2}(M^n;\mathbb{Z})$ denote the image of the fundamental class of $N$, and let $e\in H^2(M^n;\mathbb{Z})$ denote its Poincar\'e dual.
				By Lemma \ref{lem:periodic}, the homomorphisms $\cup e: H^j(M^n;\mathbb{Z})\to H^{j+2}(M^n;\mathbb{Z})$ are surjective for $1\leq j< n-3$ and injective for $1< j\leq n-3$.
				Because $M^n$ is simply connected, $H^{1}(M^n;\mathbb{Z})\cong 0\cong H^{n-1}(M^n;\mathbb{Z})$, and it follows that $H^i(M^n;\mathbb{Z})=0$ for all odd values of $i$.
			\end{proof}
			
			\begin{example}
				As we described in Example \ref{ex:S3xS3}, $S^3\times S^3$ admits a metric with $\Ric_2>0$ invariant under an effective $T^3$-action.
				It follows from Lemma \ref{lem:codim2even} that no action of any of the circle subgroups of $T^3$ has fixed point set of codimension $2$.
				In fact, one can show that the possible fixed point sets in $S^3\times S^3$ of circle subgroups in this example are either empty or diffeomorphic to $T^2$. 
			\end{example}
			
			We will now prove Theorem \ref{thm:even3/4maxsymrank} using the work of Smale in \cite{smale} and Dessai and Wilking in \cite{DessaiWilking}:
			
			\begin{proof}[Proof of Theorem \ref{thm:even3/4maxsymrank}]
				We will prove Theorem \ref{thm:even3/4maxsymrank} by induction on the dimension, $n$. 
				First note that in dimensions $n=8,10$, and $12$, we have $\lceil \frac{3n+6}{8} \rceil = \frac{n}{2}$.
				So the result holds in these dimensions by Theorem \ref{thm:evenmaxsymrank}, thus establishing our base case.
				
				Now, suppose for the sake of induction that Theorem \ref{thm:even3/4maxsymrank} holds in even dimensions $8,\dots,n-2$ for some $n\geq 14$. 
				We will show it also holds in dimension $n$.
				Consider a manifold $M^n$ satisfying the hypotheses of Theorem \ref{thm:even3/4maxsymrank}, and suppose $T^r$ acts isometrically and effectively on $M^n$ with $r = \symrank(M^n) \geq \frac{3n+6}{8}$.
				Then by Theorem \ref{result:evenEuler}, $\chi(M^n) > 0$, and hence the $T^r$-action on $M^n$ has a fixed point.
				By Lemma \ref{lem:NRic2even}, there exists a connected submanifold $N\subset M$ of minimal codimension fixed by a circle subgroup of $T^r$ such that $\dim N\geq\frac{3n-2}{4}$ and either $\codim(N)=2$ or $\symrank(N)\geq \frac{3\dim N+6}{8}$. 
				
				If $\codim(N)=2$, then by Theorem \ref{result:connprinc}, the inclusion $N^{n-2}\hookrightarrow M^n$ is $(n-3)$-connected, and by Lemma \ref{lem:codim2even}, $H^\mathrm{odd}(M^n;\mathbb{Z}) = 0$.
				By Lemma \ref{lem:periodic}, if $[N]\in H_{n-d}(M^n;\mathbb{Z})$ denotes the image of the fundamental class of $N$ and $e\in H^d(M^n;\mathbb{Z})$ denotes its Poincar\'e dual, then the homomorphisms $\cup e: H^i(M;\mathbb{Z})\to H^{i+d}(M;\mathbb{Z})$ are isomorphisms for $2\leq i \leq n-4$.
				If $b_2(M^n)\leq 1$, then because $M^n$ is simply connected, it has the cohomology ring of either a sphere or a complex projective space, and hence is homotopy equivalent to one of these spaces.
				For the case of a sphere, $M^n$ is homeomorphic to $S^n$ by Smale's resolution to the Poincar\'e conjecture for dimensions $\geq 5$ in \cite{smale}.
				In the case of a complex projective space, Dessai and Wilking proved in \cite{DessaiWilking} if a manifold is homotopy equivalent to $\mathbb{C}\mathrm{P}^m$ and admits a smooth effective action by a torus $T^r$ such that $m < 4r - 1$, then the manifold is tangentially homotopy equivalent to $\mathbb{C}\mathrm{P}^n$.
				Thus, it follows in this case that $M^n$ is tangentially homotopy equivalent to $\mathbb{C}\mathrm{P}^{n/2}$.
				
				Suppose instead $\symrank(N)\geq \frac{3\dim N+6}{8}$. 
				Then because $\dim N \geq \frac{3n-2}{4} \geq 10$ and $N$ is totally geodesic in $M$, we have that $\Ric_2(N)>0$.
				Thus, the induction hypothesis implies that $H^\mathrm{odd}(N;\mathbb{Z}) = 0$.
				Furthermore, because
				\[
					n - 2\codim N + 1 \geq \tfrac{n}{2},
				\]
				the inclusion $N\hookrightarrow M^n$ is at least $\frac{n}{2}$-connected by Theorem \ref{result:connprinc}.
				Thus it follows that $H^\mathrm{odd}(M^n;\mathbb{Z}) = 0$.
				If $b_2(M^n)\leq 1$, then $b_2(N)\leq 1$, and $N$ is either homeomorphic to a sphere are tangentially homotopy equivalent to a complex projective space by the induction hypothesis.
				Because the inclusion $N\hookrightarrow M^n$ is $\frac{n}{2}$-connected, it follows that $M^n$ has the cohomology ring of either a sphere or a complex projective space.
				Then, just as in the previous case, $M^n$ is either homeomorphic to $S^n$ or tangentially homotopy equivalent to $\mathbb{C}\mathrm{P}^{n/2}$.
			\end{proof}

	\bibliographystyle{alpha}
	\bibliography{bibfile}

\end{document}